\documentclass[11pt,reqno]{preprint}

\usepackage[full]{textcomp}
\usepackage[osf]{newtxtext}
\usepackage{comment}

\usepackage{amssymb}
\usepackage{mathtools}
\usepackage{hyperref}
\usepackage{mhenvs}
\usepackage{mhequ}
\usepackage{mhsymb}
\usepackage{mathrsfs}
\usepackage{enumerate}
\usepackage{appendix}

\newop{Var}
\newop{Ric}
\newop{Cov}
\newop{supp}

\def\gref{\ensuremath{g_{\mathrm{ref}}}}
\def\qref{\ensuremath{Q_{\mathrm{ref}}}}
\def\nunqf{\ensuremath{\nu_\mathrm{NQF}}}
\def\nuqf{\ensuremath{\nu_\mathrm{LQF}}}

\title{Stochastic conformal flows in even dimensions}

\author{Jack Piazza}

\institute{University of Wisconsin-Madison, \email{japiazza2@wisc.edu}}

\begin{document}

\maketitle

\begin{abstract}
	We define two stochastic analogs of a geometric flow on even-dimensional manifolds called $Q$-curvature flow, and use the theory of Dirichlet forms to construct weak solutions to both. The first of these flows, which we call the normalized $Q$ flow (NQF), preserves the volume normalization from the deterministic setting. The second, which we call the Liouville $Q$ flow (LQF), has a different normalization motivated by a similar flow studied by Dubédat and Shen. The volume dynamics of NQF and LQF are shown to evolve as square Bessel and CIR processes, respectively. We also show that under certain additional conditions, LQF is a stochastic quantization of the even-dimensional Polyakov-Liouville measures recently introduced by Dello Schiavo, Herry, Kopfer, and Sturm. 
\end{abstract}

\setcounter{tocdepth}{2}
\tableofcontents
\section{Introduction}

Let $M$ be a closed manifold of dimension $n$ equipped with a fixed reference Riemannian metric $\gref$. A geometric flow on $M$ is an evolution of a metric on $M$ of the form \begin{align*}
	\partial_t g_t = F(t, g_t)\;, \qquad g_0 = g_i \;.
\end{align*}
One of the most well-studied geometric flows is the Ricci flow, which has equation \begin{align}\label{ricci-flow}
	\partial_t g_t = -2 \Ric_t \;, \qquad g_0 = g_i
\end{align}
where $\Ric$ denotes the Ricci curvature tensor.

\begin{remark}
	Throughout the paper we use matching subscripts to indicate that a geometric object is taken with respect to a particular metric. For example, $\Ric_t$ is the Ricci curvature tensor of $(M, g_t)$, whereas the Ricci curvature of $(M, \gref)$ is $\Ric_{\mathrm{ref}}$. One important exception to this is that we occasionally use a subscript $g$ for a generic metric $g$ with no subscript (e.g. $\Ric_g$ denotes the Ricci curvature tensor of $(M, g)$). 
\end{remark}

One case where the Ricci flow is particularly tractable is when $n=2$. This is because two-dimensional Ricci curvature has the simple form $\Ric_g = K_g g$, where $K$ is the Gauss curvature. The Ricci flow is then \begin{align}\label{ricci-flow-2d}
	\partial_t g_t = -2 K_t g_t \;, \qquad g_0 = g_i
\end{align}
and so the time derivative of $g_t$ is a multiple (over $C^\infty (M)$) of $g_t$ itself.

Recall that two metrics $g$ and $g'$ are said to be conformally equivalent if there is a function $\phi \in C^\infty (M)$, called the conformal factor, such that \[
	g' = e^{2 \phi} g \;.
\]
As the name suggests, conformal equivalence is an equivalence relation which partitions the space of metrics on $M$ into so-called conformal classes. In the special case where $g$ is conformally equivalent to $g_\mathrm{ref}$, we denote by $\phi_g$ the conformal factor for which \[
	g = e^{2 \phi_g} g_\mathrm{ref} \;.
\]
Almost all of the metrics discussed in this paper will lie in the conformal class of $g_\mathrm{ref}$. In particular, we assume that the initial condition $g_i$ of \eqref{ricci-flow-2d} is in this class.

Inspecting \eqref{ricci-flow-2d}, one sees that any solution must remain in the same conformal class as $g_i$ (and hence $g_\mathrm{ref}$) for its entire lifetime. A geometric flow with this property is called a conformal flow. Crucially, conformal flows can be recast as partial differential equations in terms of the conformal factor. For example, if we write a solution to \eqref{ricci-flow-2d} as \[
	g_t = e^{2 \phi_t} g_\mathrm{ref}
\]
then take a time derivative and rearrange, we obtain \[
	\partial_t \phi_t = -K_t \;.
\]

This is useful because it is typically much easier to analyze equations on function spaces than on spaces of tensorial objects like Riemannian metrics.

\begin{remark}
	Note that in the above equations, $\phi_t$ relates $g_t$ to $g_\mathrm{ref}$, not $g_i$. This means that $\phi_0$ is typically not identically zero.
\end{remark}

Another simplifying aspect of conformal flows is that many geometric quantities scale in straightforward ways under conformal transformations. For example, under the conformal change of metric $g = e^{2\phi_g} g_\mathrm{ref}$, the Gauss curvature $K$ scales as \begin{align}\label{gauss-conformal-inv}
	K_g = e^{-2\phi_g} (K_\mathrm{ref} - \Delta_\mathrm{ref} \phi_g)
\end{align}
where $\Delta$ is the Laplace-Beltrami operator. Any quantity with a scaling law like this is called a conformal quasi-invariant. A conformal quasi-invariant is a conformal invariant if it remains constant under conformal transformations. For example, a consequence of the Gauss-Bonnet theorem is that $\omega (K)$ is a conformal invariant in two dimensions, where $\omega$ is the volume form associated to the metric and \begin{align*}
	\omega (f) \coloneq \int_M f \, \omega
\end{align*}
whenever this integral makes sense.

One of the most fundamental properties of the Ricci flow is that, after applying a suitable normalization, its solutions converge in many cases to metrics of constant curvature. The normalization is required to ensure that the total volume, $\omega (1)$, is held constant. In two dimensions this can be done by modifying \eqref{ricci-flow-2d} to \begin{align}\label{ricci-flow-normalized}
	\partial_t g_t = -2 (K_t - \overline{K_t}) g_t
\end{align}
where $\overline{K} = \omega (K) / \omega (1)$ denotes the average of $K$ (we also adopt this notation for functions other than $K$). The solution theory of the normalized Ricci flow has been studied extensively. See \cite{Hamilton88} and \cite{Chow91} for the two-dimensional case, \cite{Hamilton82} for the three-dimensional case, and \cite{Brendle08} for a higher-dimensional result. 

The Ricci flow can be thought of as a geometric version of the heat equation, though it is considerably more complex due to its nonlinearity. Since the stochastic heat equation is one of the most well-understood stochastic PDEs, a natural question is whether the Ricci flow also has a stochastic analog. In the conformal (i.e. two-dimensional) setting, \cite{DS22} answered this question in the affirmative by showing the existence of weak solutions to a stochastic version of the Ricci flow. 

The primary aim of this paper is to construct stochastic analogs of conformal flows in higher dimensions. Since the Ricci flow is only conformal in two dimensions, we must work with a different deterministic flow, which we introduce next.

\subsection{Q-curvature}\label{q-curv}

From now on we assume that the dimension $n$ is even. $Q$-curvature is a conformally quasi-invariant function defined on even-dimensional Riemannian manifolds. It was first introduced by Branson and Ørsted in four dimensions (\cite{BransonOrsted91}) and its properties were later developed further (see for example \cite{Branson95, BCY}), including extending it to all even dimensions. In low dimensions, $Q$-curvature can be expressed via an explicit formula in terms of the Riemann curvature tensor and its derivatives. For example, in two dimensions $Q$ is just the Gauss curvature $K$. In four dimensions it has the formula \[
	Q = -\frac{1}{6} \left( \Delta R - R^2 + 3 \lvert \Ric \rvert^2 \right)
\]
where $R = \tr \Ric$ is the Ricci scalar curvature. 

The most important property of $Q$-curvature is that it satisfies an analog of \eqref{gauss-conformal-inv}: If $g = e^{2\phi_g} g_\mathrm{ref}$ then (\cite{Branson95}, Corollary 1.4) \begin{align}\label{q-conformal-inv}
	Q_g = e^{-n \phi_g} (Q_\mathrm{ref} + P_\mathrm{ref} \phi_g) \;.
\end{align}
Here $P$ is  a differential operator of order $n$ called a co-polyharmonic operator. We will define these operators and discuss their properties in Section \ref{conf-geo}; until then, one can think of $P_g$ abstractly as an operator which is self-adjoint on $L^2 (\omega_g)$ and annihilates constant functions. There is also an analog of the conformal invariant $\omega(K)$ for $Q$-curvature. Let \[
	Q(f) \coloneq \int_M Qf \, \omega
\]
whenever this integral makes sense. Then it follows from \eqref{q-conformal-inv} that $Q(1)$ is a conformal invariant: \[
	Q_g (1) = \int_M e^{-n \phi_g} (Q_\mathrm{ref} + P_\mathrm{ref} \phi_g) \, \omega_g = \int_M Q_\mathrm{ref} + \phi_g P_\mathrm{ref} 1 \, \omega_\mathrm{ref} = Q_\mathrm{ref} (1) \;.
\]

It will be useful to assume that the reference metric $\gref$ is chosen so that the $Q$-curvature $\qref$ is constant. Branson, Chang, and Yang (\cite{BCY}) showed that in dimension four, the conformal class of $g$ contains such a metric as long as the following conditions are satisfied: \begin{enumerate}
	\item[(A1)] $P_g$ is positive semi-definite with kernel given by the constant functions.
	\item[(A2)] $Q_g (1) < Q_r (1)$, where $g_r$ is the round metric on the sphere $S^n$. 
\end{enumerate}
Brendle (\cite{Brendle2003}) showed that these conditions imply the existence of such a metric in all even dimensions. Condition (A2) was later weakened in \cite{Ndiaye07}, but we will need the full strength of (A2) for our purposes so we use the stronger version here. Since $Q(1)$ is a conformal invariant and $P$ is a conformal quasi-invariant (see Section \ref{conf-geo}), both (A1) and (A2) are class properties: they either hold for all metrics in a conformal class or for none of them. Unless otherwise stated, we always assume that there is a metric on $M$ such that (A1) and (A2) hold, and choose $\gref$ in the same conformal class so that $\qref$ is constant. We also assume throughout that $(M, g_\mathrm{ref})$ is locally conformally flat, meaning every point has a neighborhood on which $g_\mathrm{ref}$ is conformally equivalent to a metric with vanishing Riemann curvature tensor. We will discuss the geometric meaning of these conditions and classes of manifolds which satisfy them in Section \ref{topology}. 

\subsection{The Q flow}

$Q$-curvature can be used to define a conformal flow on even-dimensional manifolds. This flow is called the $Q$-curvature flow, or $Q$ flow for short. Just as solutions to the normalized 2D Ricci flow converge to metrics of constant Gauss curvature, solutions to the $Q$ flow are expected to converge to metrics with $Q$-curvature proportional to some pre-specified ``prescribing" function $f \in C^\infty (M)$. More precisely, Brendle (\cite{Brendle2003}) showed that if $f > 0$ and conditions (A1) and (A2) hold, the equation \begin{align}\label{nqf}
	\partial_t g_t = -2 \left( Q_t - \frac{Q_t (1)}{\omega_t (f)} f \right) g_t \;, \qquad g_0 = g_i
\end{align}
has a global-in-time solution and converges to a metric $g_\infty$  such that \[
	Q_\infty = \frac{Q_\infty (1)}{\omega_\infty (f)} f \;.
\]
The fraction preceding $f$ in \eqref{nqf}, whose denominator is always nonzero by the positivity of $f$, ensures that the total volume is preserved. Indeed, the associated equation for $\phi_t$ is \begin{align}\label{nqf-phi}
	\partial_t \phi_t = -\left( Q_t - \frac{Q_t (1)}{\omega_t (f)} f \right) \;.
\end{align}
It follows from dominated convergence that
\begin{align*}
	&\partial_t \omega_t (1) = \int_M \partial_t e^{n \phi_t} \, \omega_\mathrm{ref} = \int_M -n \left( Q_t - \frac{Q_t (1)}{\omega_t (f)} f \right) e^{n \phi_t} \, \omega_\mathrm{ref} \\
	&= -n \int_M \left( Q_t - \frac{Q_t (1)}{\omega_t (f)} f \right) \, \omega_t = -n \left(Q_t (1) - Q_t (1) \frac{\omega_t (f)}{\omega_t (f)} \right) = 0 \;.
\end{align*}
For this reason, we call the flow associated to \eqref{nqf} the normalized $Q$ flow, or NQF. We can also consider a flow with equation \begin{align}\label{qf}
	\partial_t g_t = -2 (Q_t - f) g_t \;, \qquad g_0 = g_i
\end{align} 
where $f$ no longer needs to be positive. We refer to this as the Liouville $Q$ flow, or LQF. The corresponding equation for $\phi$ is \begin{align}\label{qf-phi}
	\partial_t \phi_t = -\left( Q_t - f \right) \;.
\end{align}

The above equations for NQF and LQF may seem dissimilar to the Ricci flow because of the presence of the function $f$. However, if $f$ is a positive constant in NQF, \eqref{nqf} becomes
 \begin{align}\label{nqf-constant}
		\partial_t g_t = -2 (Q_t - \overline{Q_t}) g_t \;.
\end{align}
which is analogous to \eqref{ricci-flow-normalized}. Similarly, if $f = \overline{Q_i}$, \eqref{qf} becomes \begin{align}\label{qf-constant}
		\partial_t g_t = -2 (Q_t - \overline{Q_i}) g_t \;.
\end{align}
Since NQF preserves both volume and the conformal invariant $Q_t (1)$, it preserves $\overline{Q_t}$. A similar computation shows that for the particular choice of $f$ in \eqref{qf-constant},
\[
	\partial_t \omega_t (1) = n \overline{Q_i} (\omega_t (1) - \omega_i (1))
\]
so \eqref{qf-constant} also preserves volume and hence $\overline{Q_t}$ (note that this is not true more generally for \eqref{qf-phi}). It follows that \eqref{nqf-constant} and \eqref{qf-constant} have the same solutions. In particular, the normalized 2D Ricci flow is a special case of both NQF and LQF.  

\begin{remark}\label{volume-rmk}
	The equivalence between \eqref{nqf-constant} and \eqref{qf-constant} relies on the fact that both are volume-preserving. In what follows we will consider stochastic analogs of NQF and LQF for which volume is no longer preserved. Thus, the stochastic versions of \eqref{nqf-constant} and \eqref{qf-constant} will no longer have the same solutions even though their deterministic counterparts do.
	
	For a simple example of this phenomenon, consider the equations $dX_t = 0$ and $dX_t = (X_t - X_0) \, dt$ for a real-valued process $X$. Though they have the same (constant) solutions, adding a noise term $dB_t$ to each produces equations with drastically different solutions. We must therefore be careful to take note of which properties will fail to transfer to the stochastic setting.
\end{remark}

NQF and LQF can both be expressed as gradient flows. Let $\CM_0$ be the space $(C^\infty (M), \mathbf{g})$, where $\mathbf{g}$ is the Calabi metric defined by \[
	\mathbf{g}_{\phi} (h_1, h_2) = \int_M h_1 h_2 e^{n \phi} \, \omega_{\mathrm{ref}} \;.
\]
Note that since $C^\infty (M)$ is infinite-dimensional, $\mathbf{g}$ is a purely formal Riemannian structure which does not turn $\CM_0$ into a Riemannian manifold. Regardless, this choice of $\mathbf{g}$ is natural from a geometric perspective because at the point $\phi_g$, it is the $L^2$ inner product associated to the volume form of $g = e^{2\phi_g} \gref$. Consider the following two functionals on $\CM_0$: \begin{align}\label{nqf-functional}
	E^1_f [\phi] = \int_M \frac{1}{2} \phi P_{\mathrm{ref}} \phi \, \omega_{\mathrm{ref}} + \int_M Q_\mathrm{ref} \phi \, \omega_{\mathrm{ref}} - \frac{1}{n} Q_\mathrm{ref} (1) \log \left( \int_M e^{n \phi} f \, \omega_\mathrm{ref} \right)
\end{align}
and \begin{align}\label{qf-functional}
	E^2_f [\phi] = \int_M \frac{1}{2} \phi P_{\mathrm{ref}} \phi \, \omega_{\mathrm{ref}} + \int_M Q_\mathrm{ref} \phi \, \omega_{\mathrm{ref}} - \frac{1}{n} \int_M e^{n \phi} f \, \omega_\mathrm{ref} \;.
\end{align}
\begin{proposition}
	The gradient flows for the functionals $E^1_f$ and $E^2_f$ on $\CM_0$ are precisely the flows \eqref{nqf-phi} and \eqref{qf-phi} for the conformal factor $\phi$ in NQF and LQF respectively. 
\end{proposition}
\begin{remark}
	At least for NQF this fact is well-known (\cite{Brendle2003}). Nevertheless we include the details here because, as noted in Remark \ref{volume-rmk}, it is important that the calculations do not rely on the fact that the flows preserve the total volume.
\end{remark}
\begin{proof}
	We compute the directional derivatives of $E^1_f$ and $E^2_f$ at a point $\phi_g \in C^\infty (M)$, which we think of as the conformal factor for a metric $g = e^{2\phi_g} g_\mathrm{ref}$, in the direction of $h \in C^\infty (M)$. Starting with the term which is quadratic in $\phi_g$, \[
		\int_M \frac{1}{2} (\phi_g +\eps h) P_\mathrm{ref} (\phi_g + \eps h) \, \omega_{\mathrm{ref}} - \int_M \frac{1}{2} \phi_g P_\mathrm{ref} \phi_g \, \omega_\mathrm{ref} = \eps \int_M h P_\mathrm{ref} \phi_g \, \omega_\mathrm{ref} + O(\eps^2) 
	\]
	using the self-adjointness of $P_\mathrm{ref}$. For the linear term, \[
	\int_M Q_\mathrm{ref} (\phi_g + \eps h) \, \omega_\mathrm{ref} - \int_M Q_\mathrm{ref} \phi_g \, \omega_\mathrm{ref} = \eps \int_M h Q_\mathrm{ref} \, \omega_\mathrm{ref} + O(\eps^2) \;.
	\]
	Next we handle the logarithmic term appearing in $E^1_f$: \begin{align*}
		&\log \left( \int_M e^{n (\phi_g + \eps h)} f \, \omega_\mathrm{ref} \right) - \log \left( \int_M e^{n \phi_g} f \, \omega_\mathrm{ref} \right) \\
		&= \log \left( 1 + \frac{\int_M e^{n \phi_g} (e^{\eps n h} - 1) f \, \omega_\mathrm{ref}}{\int_M e^{n \phi_g} f \, \omega_\mathrm{ref}} \right)  \\
		&= \log \left( 1 + \frac{\int_M e^{n \phi_g} (\eps n h + O(\eps^2)) f \, \omega_\mathrm{ref}}{\int_M e^{n \phi_g} f \, \omega_\mathrm{ref}} \right)  \\
		&= \frac{\int_M e^{n \phi_g} (\eps n h) f \, \omega_\mathrm{ref}}{\int_M e^{n \phi_g} f \, \omega_\mathrm{ref}} + O(\eps^2) \\
		&= \eps \left( n \frac{\int_M hf \, \omega_g}{\omega_g (f)} \right) + O(\eps^2) \;.
	\end{align*}
	These three computations are enough to find the directional derivative of $E^1_f$: \begin{align*}
		&\lim_{\eps \to 0} \frac{E^1_f [\phi_g + \eps h] - E^1_f [\phi_g]}{\eps} = \int_M h(P_\mathrm{ref} \phi_g + Q_\mathrm{ref}) \, \omega_\mathrm{ref} - \frac{Q_\mathrm{ref} (1)}{\omega_g (f)} \int_M hf \, \omega_g \\
		&= \int_M h Q_g \, \omega_g - \frac{Q_g (1)}{\omega_g (f)} \int_M hf \, \omega_g = \left\langle h, Q_g - \frac{Q_g (1)}{\omega_g (f)} f \right\rangle_{L^2 (\omega_g)}
	\end{align*}
	which (up to sign) matches \eqref{nqf-phi}.
	
	For the last term in $E^2_f$ we compute 	\begin{align*}
		&\int_M e^{n (\phi_g + \eps h)} f \, \omega_\mathrm{ref} - \int_M e^{n \phi_g} f \, \omega_\mathrm{ref} = \int_M e^{n \phi_g} (e^{\eps n h} - 1) f \, \omega_\mathrm{ref} \\
		&= \int_M e^{n \phi_g} (\eps n h + O(\eps^2)) f \, \omega_\mathrm{ref} = \eps \left( n \int_M hf \, \omega_g \right)+ O(\eps^2)
	\end{align*}
	so the directional derivative is \begin{align*}
		&\lim_{\eps \to 0} \frac{E^2_f [\phi_g + \eps h] - E^2_f [\phi_g]}{\eps} = \int_M h (P_\mathrm{ref} \phi_g + Q_\mathrm{ref}) \, \omega_\mathrm{ref} - \int_M hf \, \omega_g \\
		&= \int_M h (Q_g - f) \, \omega_g = \langle h, Q_g - f \rangle_{L^2 (\omega_g)}
	\end{align*}
	which (up to sign) matches \eqref{qf-phi}.
\end{proof}

\subsection{Langevin flow}

With NQF and LQF at hand, our next goal is to describe stochastic perturbations of them. This will involve adding a singular noise term to the respective equations for $\phi$, which means we no longer expect $\phi$ to be smooth at any fixed time. For now we will ignore these concerns and treat everything as though it is smooth; in Section \ref{conf-geo} we will define all of the relevant objects more precisely.

Let us first describe a general procedure for constructing a stochastic dynamic from a gradient flow. Let $(X, g)$ be a closed manifold and consider the gradient flow $(x_t)_{t \geq 0}$ with respect to a potential $V: X \to \R$. The generator for this flow is \[
	-\nabla_g V \cdot \nabla_g \;.
\]
The Langevin flow associated to $V$ is a stochastic perturbation with generator \[
	\tfrac{\sigma^2}{2} \Delta_g - \nabla_g V \cdot \nabla_g 		
\]
where $\sigma > 0$. 

When $X$ is finite-dimensional, this flow satisfies the SDE \[
	dx_t = -\nabla_g V (x_t) \, dt + \sigma \, dB_t
\] 
where $B$ is a Brownian motion on $(X, g)$. See \cite{Hsu02} for a precise interpretation of equations of this type on manifolds. Moreover, it has an invariant measure with density proportional to $e^{-2V / \sigma^2} \, d\omega$. 

Unfortunately, these two facts do not hold in general when $X$ is infinite-dimensional. However, we can still make sense of the differential equation and invariant measure in some cases. The strategy is to first define the measure by itself, then use Dirichlet form techniques to construct a process for which it is invariant (or sometimes just symmetrizing). Finally, one can show that this process solves the desired differential equation.

Let us now return to the setting of the $Q$ flow. The formal equation for the NQF Langevin flow is \begin{align}\label{stoch-nqf-phi}
	\partial_t \phi_t = - \left(Q_t - \frac{Q_t (1)}{\omega_t (f)} f \right) + \sigma \xi_t
\end{align}
where $\xi_t$ is a spacetime white noise. Here ``spatially white" means $\xi_t$ is white with respect to the metric $g_t$. For LQF the formal equation is \begin{align}\label{stoch-qf-phi}
	\partial_t \phi_t = -(Q_t - f) + \sigma \xi_t \;.
\end{align}
In what follows, when we say NQF or LQF we are referring to these stochastic flows rather than their deterministic counterparts. 

As in the deterministic setting, these equations yield corresponding equations for the metric and the volume form. For NQF these are \begin{align}\label{stoch-nqf}
	\partial_t g_t = -2 \left( Q_t - \frac{Q_t (1)}{\omega_t (f)} f \right) g_t + 2 \sigma \xi_t g_t \;,
\end{align}
\begin{align}\label{stoch-nqf-vol}
	\partial_t \omega_t = -n \left( Q_t - \frac{Q_t (1)}{\omega_t (f)} f \right) \omega_t + n \sigma \xi_t \omega_t \;.
\end{align}
For LQF they are \begin{align}\label{stoch-qf}
	\partial_t g_t = -2 \left( Q_t - f \right) g_t + 2 \sigma \xi_t g_t \;,
\end{align}
\begin{align}\label{plain-stoch-qf-vol}
	\partial_t \omega_t = -n \left( Q_t - f \right) \omega_t + n \sigma \xi_t \omega_t \;.
\end{align}

As previously mentioned, in this section we treat everything as though it is smooth so we do not worry about the meaning of products like $\xi_t g_t$. For technical reasons related to the Dirichlet form methods we employ, we mostly focus on the volume form equations \eqref{stoch-nqf-vol} and \eqref{plain-stoch-qf-vol}. In fact, we will see in Section \ref{conf-GMC} that there is an equivalence between $\phi_t$ and $\omega_t$ which allows us to pass between solutions to these equations.

Let us also record the formal expression that we expect to see for the invariant measures of these Langevin flows. Denote by $\CM$ the space of nonzero positive finite Borel measures on $M$ with the topology of weak convergence. Note that $\CM$ is homeomorphic to $\CP \times (0, \infty)$, where $\CP$ is the space of Borel probability measures on $M$, so in particular $\CM$ is locally compact. From the functional $E^1_f$, we see that for NQF we expect an invariant measure on $\CM$ with formal density proportional to \begin{align}\label{nqf-measure}
	\omega (f)^{2 Q_\mathrm{ref} (1) / (n \sigma^2)} \exp \left( - \sigma^{-2} \omega_\mathrm{ref} (\phi P_\mathrm{ref} \phi + 2 Q_\mathrm{ref} \phi) \right) \, \omega_\mathbf{g} (d\omega)
\end{align}
where $\phi$ is the conformal factor corresponding to $\omega$, and $\omega_\mathbf{g}$ is thought of as a volume form on $\CM$ associated to the Calabi metric. In reality, the volume form $\omega_\mathbf{g}$ does not exist since $\CM$ is infinite-dimensional, so we will need a way to precisely interpret this measure. The corresponding measure for LQF has formal density proportional to \begin{align}\label{plain-qf-measure}
	\exp \left( - \sigma^{-2} \omega_\mathrm{ref} (\phi P_\mathrm{ref} \phi + 2 Q_\mathrm{ref} \phi) + 2 (n \sigma^2)^{-1} \omega (f) \right) \, \omega_\mathbf{g} (d\omega) \;.
\end{align}

Since NQF has a more intrinsic normalization than LQF, our primary motivation for studying NQF is to construct a natural stochastic analog of the normalized $Q$ flow. However, LQF also has an important purpose. We will see that it is closely linked to the Polyakov-Liouville measures for even-dimensional manifolds studied in \cite{DHKS}. In order to fully explore this connection, we will need to slightly generalize the LQF equation. We consider the adjusted volume form equation \begin{align}\label{stoch-qf-vol}
	\partial_t \omega_t = -n \left(P_\mathrm{ref} \phi_t + \rho Q_\mathrm{ref}  \right) \omega_\mathrm{ref} + nf \omega_t + n \sigma \xi_t \omega_t
\end{align}
where $\rho \geq 1$ is a new parameter. Note that if $\rho = 1$ then this is the same as \eqref{plain-stoch-qf-vol}, as all we have done is apply \eqref{q-conformal-inv}. The corresponding invariant measure should be 
\begin{align}\label{qf-measure}
	\exp \left( - \sigma^{-2} \omega_\mathrm{ref} (\phi P_\mathrm{ref} \phi + 2\rho Q_\mathrm{ref} \phi) + 2 (n \sigma^2)^{-1} \omega (f) \right) \, \omega_\mathbf{g} (d\omega) \;.
\end{align}
We will discuss the role of $\rho$ and its connection to Liouville quantum gravity in Section \ref{lqg}.

\subsection{Weak solutions}

In order to state our main result, we must define a notion of weak solution for the NQF and LQF volume form equations. We interpret weak solutions in the usual PDE sense, meaning a solution must satisfy one-dimensional projected equations obtained by pairing the original equation with smooth test functions. Pairing \eqref{stoch-nqf-vol} with some $h \in C^\infty (M)$ yields  \begin{equation}\label{nqf-sde}
\begin{aligned}
	d\omega_t (h) = &-n \left( \omega_\mathrm{ref} (h P_\mathrm{ref} \phi_t + Q_\mathrm{ref} h) - \frac{Q_t (1)}{\omega_t (f)} \omega_t (fh) \right) \, dt \\
	&+ n \sigma \| h\|_{L^2 (\omega_t)} \, dB_t
\end{aligned}
\end{equation}
where $B$ is a standard Brownian motion which can be different for different choices of $h$. Here we use the Itô isometry to rewrite $\omega_t (\xi_t h) \, dt$ as $\| h \|_{L^2 (\omega_t)} \, dB_t$. Doing the same to \eqref{stoch-qf-vol} gives \begin{align}\label{qf-sde}
	d\omega_t (h) = -n ( \omega_\mathrm{ref} (h P_\mathrm{ref} \phi_t + \rho Q_\mathrm{ref} h) - \omega_t (fh))  \, dt + n \sigma \| h\|_{L^2 (\omega_t)} \, dB_t \;.
\end{align}

$\CM$ is almost a suitable choice for the state space of a weak solution, but we must modify it slightly to account for the fact that the process could be killed by shrinking to zero or blowing up to infinity in finite time. Augment $\CM$ by adding a cemetery state $\delta$ to form $\CM_\delta = \CM \cup \{ \delta \}$. Let $\CB$ be the Borel $\sigma$-algebra on $\CM$. Following \cite{FOT11}, we say a quadruple $\mathbf{M} = (\Omega, \CF, (\omega_t)_{t \in [0, \infty]}, (P_z)_{z \in \CM_\delta})$ is a Markov process on $(\CM, \CB)$ with time parameter $t \in [0, \infty]$ if: \begin{enumerate}
	\item $P_z (\omega_t \in B)$ is $\CB$-measurable in $z \in \CM$ for each $t \geq 0$ and $B \in \CB$.
	\item There exists a filtration $\CF_t$ such that $\omega_t \in \CF_t$ and $P_z (\omega_{t+s} \in B \vert \CF_t) = P_{\omega_t} (\omega_s \in B)$ $P_z$-almost surely for all $z \in \CM$, $t, s \geq 0$, and $B \in \CB$. 
	\item $P_z (\omega_0 = z) = 1$ for all $z \in \CM$.
	\item $\omega_\infty (\tau) = \delta$ for all $\tau \in \Omega$.
	\item $P_\delta (\omega_t = \delta) = 1$ for all $t \geq 0$.
\end{enumerate}

\begin{definition}\label{weak-sol}
	A Markov process $(\Omega, \CF, (\omega_t)_{t \in [0, \infty]}, (P_z)_{z \in \CM_\delta})$ with state space $\CM$ is a \emph{weak solution to NQF (resp. LQF)} if for almost every initial condition $z \in \CM$, $(\omega_t)_{t \in [0, \infty]}$ satisfies \eqref{nqf-sde} (resp. \eqref{qf-sde}) under $P_z$ for all $h \in C^\infty (M)$.
\end{definition}

The meaning of ``almost every $z \in \CM$" is not yet clear, and will be made precise via the proof of the main theorem in Section \ref{dirichlet}.

For LQF, we will need a slightly stronger version of the condition (A2): \begin{align*}
	\text{(A2') $Q_g (1) < \rho^{-1} Q_r (1)$, where $g_r$ is the round metric on the sphere $S^n$.}
\end{align*}

Since Q flow and Ricci flow coincide in two dimensions, one can express the result of \cite{DS22} as showing the existence of a weak solution to LQF when $n=2$ and the prescribing function $f$ is $\overline{Q_\mathrm{ref}}$. Our main result generalizes this: \begin{theorem}\label{main-theorem}
	Let $(M, g)$ be a closed, locally conformally flat manifold of even dimension $n$ such that conditions (A1) and (A2) are satisfied. Choose $g_\mathrm{ref}$ in the conformal class of $g$ so that $Q_\mathrm{ref}$ is constant, and suppose $\sigma^2 < 2n^{-1} (4\pi)^{n/2} (n/2-1)!$. Then for any positive $f \in C^\infty (M)$, there exists a weak solution to NQF with prescribed $Q$-curvature $f$. If instead $f \leq 0$ and condition (A2) is replaced by (A2'), then there exists a weak solution to LQF with prescribed $Q$-curvature $f$. 
\end{theorem} 

As an immediate application of the theorem and the definition of a weak solution with $h=1$, we obtain the following corollary.

\begin{corollary}\label{volume-corollary}
	Suppose all the assumptions of Theorem \ref{main-theorem} are satisfied. The total volume $V_t \coloneq \omega_t (1)$ of the weak solutions in the theorem satisfy the SDE \[
		dV_t = n \sigma \sqrt{V_t} \, dB_t
	\]
	for NQF and \[
		dV_t = -n (\rho Q_\mathrm{ref} (1) - \omega_t (f)) \, dt + n \sigma \sqrt{V_t} \, dB_t
	\]
	for LQF. In the particular case where $Q_\mathrm{ref} \leq 0$ and $f = Q_\mathrm{ref}$, the LQF volume equation is \[
		dV_t = -n Q_\mathrm{ref} (1) \left(\rho - \tfrac{V_t}{V_\mathrm{ref}}\right) \, dt + n \sigma \sqrt{V_t} \, dB_t \;.
	\]
\end{corollary}

The processes we construct in the proof of Theorem \ref{main-theorem} will always be symmetric with respect to the associated measure \eqref{nqf-measure} or \eqref{qf-measure}. These measures will not always be invariant with respect to the corresponding process, but we can show that the measure for LQF is invariant under some additional assumptions:

\begin{corollary}\label{invariant-measure}
	Suppose all the assumptions of Theorem \ref{main-theorem} are satisfied. If $Q_\mathrm{ref} < 0$, $f = Q_\mathrm{ref}$, and $\sigma^2 \leq -2Q_\mathrm{ref} (1) / n$, then the weak solution to LQF constructed in the proof of Theorem \ref{main-theorem} has an invariant measure with formal density given by \eqref{qf-measure}. In the special case where $\rho = 1 + \frac{n \sigma^2}{2 (n/2 - 1)! (4\pi)^{n/2}}$, this invariant measure is the conformally quasi-invariant adjusted Polyakov-Liouville measure constructed in \cite{DHKS}.
\end{corollary}

Let us briefly outline the remainder of the paper. In Section 2 we develop the definitions needed to analyze geometric objects in the (non-smooth) stochastic setting. We use this to give rigorous meaning to the symmetrizing measures for NQF and LQF. In Section 3 we prove integration-by-parts formulas for these measures. In Section 4 we use the theory of Dirichlet forms to construct processes associated to the measures. We then show that these processes are weak solutions to the NQF and LQF dynamics, proving Theorem \ref{main-theorem}. Finally, in Section 5 we prove Corollaries \ref{volume-corollary} and \ref{invariant-measure}, explore the connection between LQF and Liouville quantum gravity, and discuss the topological conditions needed for our results. A technical lemma needed for the constructions in Section 2 is proved in Appendix A. 

The core of the argument is similar to \cite{DS22} and other existence proofs using Dirichlet forms; the main novelties come from the technical geometric and topological difficulties that arise when working with random objects on higher-dimensional manifolds. This will be discussed further in the next section.

\begin{acknowledge}
	The author is very grateful to Hao Shen for suggesting the problem and for many helpful discussions, and to David Clancy for pointing out that the LQF volume dynamic is a CIR process. The author was supported in part by NSF grant DMS- 2037851.
\end{acknowledge}

\section{Construction of the symmetrizing measures}

The main goal of this section is to give rigorous meaning to measures with densities \eqref{nqf-measure} and \eqref{qf-measure}. To do this, we must first redefine several quantities in a general setting where $\phi$ is no longer assumed smooth. We will then give an overview of some canonical random objects associated to the manifold $M$: namely, co-polyharmonic Gaussian fields (CGFs) and co-polyharmonic Gaussian multiplicative chaos (CGMC) measures. Finally, we will use these objects to define the measures for NQF and LQF.

\subsection{Conformal geometry in the stochastic setting}\label{conf-geo}

Recall that in Section \ref{q-curv} we made use of the so-called co-polyharmonic operators $P_g$. The following proposition defines these operators along with some of their properties.

\begin{proposition}\label{copoly-ops}
	Let $M$ be a closed manifold of even dimension $n$. There is a family of operators $P_g: C^\infty (M) \to C^\infty (M)$, indexed by metrics on $M$, such that \begin{enumerate}[(i)]
		\item $P_g$ is a differential operator of order $n$. 
		\item The leading-order term of $P_g$ is $(-\Delta_g)^{n/2}$, and there is no zeroth-order term.
		\item If $g = e^{2\phi_g} g_\mathrm{ref}$ for $\phi_g \in C^\infty (M)$ then $P_g = e^{-n \phi_g} P_\mathrm{ref}$.
		\item $P_g$ is symmetric with respect to the $L^2 (\omega_g)$ inner product.
	\end{enumerate}
\end{proposition}
Recall that metrics we work with in this paper are conformally equivalent to $g_\mathrm{ref}$, which satisfies condition (A1), so each $P_g$ is also non-negative with kernel given by the set of constant functions. These operators were originally constructed in \cite{CJMS92}. See Section 1.2 of \cite{DHKS} and the references therein for proofs of these properties as well as examples for some specific manifolds. Since $P_g$ is symmetric, it has a canonically defined Friedrichs extension which turns it into a non-negative self-adjoint operator on $L^2 (\omega_g)$.

Since the random objects we will consider are not smooth (and in fact may only be distribution-valued), we need a notion of regularity which respects the conformal geometry of $(M, g)$. We will use the following modified Sobolev spaces defined in \cite{DHKS}:

\begin{definition}
	For $s \geq 0$, the \emph{usual Sobolev space} on $(M, g)$ is $\CH^s_g = (1 - \Delta_g)^{-s/2} L^2 (\omega_g)$ with norm $\|(1 - \Delta_g)^{s/2} (\cdot ) \|_{L^2 (\omega_g)}$. For $s < 0$, it is the completion of $L^2 (\omega_g)$ with respect to the same norm. On the other hand, for $s \geq 0$ the \emph{co-polyharmonic Sobolev space} on $(M, g)$ is $H^s_g = (1 + p_g)^{-s/n} L^2 (\omega_g)$, where $p_g = a_n P_g$ is the normalized co-polyharmonic operator with normalizing constant $a_n = \tfrac{2}{(n/2-1)! (4\pi)^{n/2}}$. It has norm $\| (1 + p_g)^{s/n} (\cdot) \|_{L^2 (\omega_g)}$. If $s < 0$ then $H^s_g$ is the completion of $L^2 (\omega_g)$ with respect to the same norm. 
	
	We denote by $\mathring{\CH}^s_g$ and $\mathring{H}^s_g$ the corresponding \emph{grounded Sobolev spaces}, the subspaces of elements with zero $\omega_g$-mean. 
\end{definition}

It turns out that $\CH^s_g$ and $H^s_g$ are very similar spaces. Indeed, for any $s \in \R$ they are equal as sets and their norms are equivalent (\cite{DHKS}, Lemma 2.15). In particular, this implies that $C^\infty (M)$ is dense in $H^s_g$ for all $s \in \R$. Furthermore, from the definition of the co-polyharmonic Sobolev space, $P_g$ can be defined for any $s \in \R$ as a bounded operator from $H^s_g$ to $H^{s-n}_g$.

With these operators defined, we can now make sense of $Q$-curvature in the case where the conformal factor $\phi$ is not smooth. A standard setup is that we have a random measure $\omega_t$ and an associated conformal factor $\phi_t$ such that $\phi_t \in \mathring{H}^{-\eps}_\mathrm{ref}$ almost surely for all $\eps > 0$. In this case, we would like a definition for quantities of the form $Q_t (h)$ where $h$ is sufficiently smooth.

Even in low dimensions, the explicit formula for $Q$-curvature is now an insufficient definition because we haven't defined an extension of the Ricci curvature tensor to this setting. However, we can still use the conformal quasi-invariance of $Q$ to extend the definition. Recall that if $g = e^{2 \phi_g} g_\mathrm{ref}$ for a smooth $\phi$ then \[
	Q_g = e^{-n \phi_g} (Q_\mathrm{ref} + P_\mathrm{ref} \phi_g) \;.
\] 
Multiplying by a smooth function $h$ and integrating against $\omega_g$ gives \[
	Q_g (h) = \omega_\mathrm{ref} (Q_\mathrm{ref} h + \phi_g P_\mathrm{ref} h) \;.
\]
This formula still makes sense when $\phi$ is not smooth. Even when $\phi_g$ has regularity ``just below zero" as above, the right-hand side still makes sense provided $h \in H^s_\mathrm{ref}$ for some $s > n$. Therefore, this formula specifies $Q_g$ as an element of $H^{-n-\eps}_\mathrm{ref}$ for any $\eps > 0$. From here on we treat this as the definition of $Q_g$. Observe that with this definition, we still have the property that $Q_g (1) = Q_\mathrm{ref} (1)$ whenever $\phi_g$ has the above regularity because $P_\mathrm{ref} (1) = 0$.

\subsection{Canonical co-polyharmonic Gaussian objects}\label{conf-GMC}

Next we will see how to make sense of the equation $\omega_g = e^{n \phi_g} \omega_\mathrm{ref}$ when $\phi_g$ has low regularity, as well as how to recover $\phi_g$ from $\omega_g$ when this equation holds. We start by defining a conformally quasi-invariant analog of a log-correlated field on $M$.

Let us note a few more properties of the normalized co-polyharmonic operators $p_g$; see Section 2 of \cite{DHKS} for proofs. When viewed as an operator from $\mathring{H}^{n+s}_g$ to $\mathring{H}^s_g$ for some $s \in \R$, $p_g$ is bounded with bounded inverse. When $s=0$, the inverse has a symmetric integral kernel $k_g$ called the co-polyharmonic Green kernel. It has logarithmic growth near the diagonal: \[
	\left\lvert k_g (x, y) - \log \frac{1}{d_g (x, y)} \right\rvert \leq C
\]
uniformly over $x, y \in M$, where $C$ depends only on $M$ and $g$. We will choose $k_g$ as the covariance kernel for a random field on $M$.

\begin{definition}[\cite{DHKS}, Section 3]\label{CGF}
	Let $s > 0$ and let $(M, g)$ satisfy (A1). A \emph{co-polyharmonic Gaussian field (CGF)} on $(M, g)$ is a centered Gaussian distribution $\psi$ in $\mathring{H}^{-s}_g$ with covariance \[
		\E [(\psi, u) (\psi, v)] = \int_M \int_M k_g (x, y) u(x) v(y) \, \omega_g (dx) \, \omega_g (dy)
	\]
	for all $u, v \in \mathring{H}^s_g$. Such a field exists and is unique in distribution. Moreover, the choice of $s$ does not matter, since any two such fields with different choices of $s$ have the same distribution. Denote the law of a CGF on $(M, g)$ by $\mu_g$. 
\end{definition}

Tensoring $\mu_g$ with Lebesgue measure on $\R$ then taking the pushforward under the map $(\phi, c) \mapsto \phi + c$ yields a $\sigma$-finite measure $\tilde{\mu}_g$, the ``law" of an ungrounded CGF on $(M, g)$. An advantage of $\tilde{\mu}_g$ is that it is conformally invariant, i.e. it does not depend on the choice of metric within a conformal class (\cite{DHKS}, Proposition 3.16). 

Now we consider the expressions $e^{\gamma \phi_g} \omega_\mathrm{ref}$ when $\gamma \in \R$ and $\phi_g \in \mathring{H}^{-\eps}_\mathrm{ref}$ for all $\eps > 0$. These are co-polyharmonic analogs of the Gaussian multiplicative chaos (GMC) measures originally studied by Kahane (\cite{Kahane85}, see also \cite{BP24} for an introduction). For $r \in \R$, let $\mathring{H}^{r-}_g$ denote the set of distributions which are in $\mathring{H}^{r-s}_g$ for all $s > 0$. For example, a CGF lies in $\mathring{H}^{0-}_g$ almost surely.

\begin{proposition}[\cite{DHKS} Theorem 4.1]\label{CGMC}
Suppose $\gamma \in [0, \sqrt{2n})$ and $(M, g)$ satisfies (A1). There is a measurable map \[
	M^\gamma_g: \mathring{H}^{0-}_g \to \CM
\]
satisfying the following properties: \begin{enumerate}[(i)]
	\item For $\mu_g$-a.e. $\psi$ and every $h \in \mathring{H}^{n/2}_g$, $M^\gamma_g (\psi+h) = e^{\gamma h} M^\gamma_g (\psi)$.
	\item For all $p \in (-\infty, \frac{2n}{\gamma^2})$, $\E_{\mu_g} [(M^\gamma_g (\psi) (1))^p] < \infty$.
\end{enumerate}
$M^\gamma_g (\psi)$ is called a \emph{co-polyharmonic Gaussian multiplicative chaos (CGMC) measure}. The map $M^\gamma_g$ extends to $H^{0-}_g$ by defining $M^\gamma_g (\psi + c) = e^{\gamma c} M^\gamma_g (\psi)$ for any constant $c$.
\end{proposition} 

In other words, we can obtain a measure $\omega_g = e^{\gamma \phi_g} \, \omega_\mathrm{ref}$ from $\phi_g$ so long as the regularity of $\phi_g$ is $0-$. We also want to be able to recover $\phi_g$ from $\omega_g$, so we need a measurable inverse to the map $M^\gamma_g$ (one can think of this map as taking the ``logarithm" of a GMC measure). For a log-correlated field $G$ on a bounded domain $D_0 \subseteq \R^d$, let $e^{\gamma G} \omega$ denote the Euclidean GMC measure with ground measure $\omega$ on $D_0$. We use the following result of Vihko:

\begin{proposition}[\cite{Vihko24}]\label{gmc-inversion}
	Suppose $\gamma \in (0, \sqrt{2n}]$ and $G$ is a centered Gaussian field on a bounded domain $D_0 \subseteq \R^d$ with covariance kernel of the form \[
		C_G (x, y) = \log \frac{1}{\lvert x - y \rvert} + q_G(x, y)
	\]
	where $q_G \in H^{n+\eps}_\mathrm{loc} (D_0 \times D_0)$ is continuous on the interior of $D_0$. Then for any $D$ compactly contained in $D_0$, there is a map $X^\gamma$ from the space of Borel measures on $D$ to the space of distributions on $D$ such that $X^\gamma (e^{\gamma G} m) = G$ almost surely, where $m$ is Lebesgue measure on $D$.
\end{proposition}

\begin{remark}\label{as-conv}
	The way we state this result here is slightly different from the way it is stated in \cite{Vihko24}. Theorem A of \cite{Vihko24} says that there are random variables $G_\eps$ which are measurable with respect to $e^{\gamma G} m$ such that $\langle \psi, G_\eps \rangle \to \langle \psi, G \rangle$ in probability as $\eps \to 0$ for all suitable test functions $\psi$. It follows from this that one can extract $G$ from $e^{\gamma G} m$ in a measurable way on compact subsets.
\end{remark}

It is not immediate that we can apply this result to our situation, since it only holds for domains in Euclidean space. The next lemma addresses this issue. Denote by $\CD' (M)$ the space of distributions on $(M, g_\mathrm{ref})$. 

\begin{lemma}\label{CGMC-inversion}
	Suppose $\gamma \in (0, \sqrt{2n})$ and $\psi$ is a CGF on $(M, g_\mathrm{ref})$. Then there is a map $X^\gamma: \CM \to \CD' (M)$ such that $X^\gamma (M^\gamma_\mathrm{ref} (\psi)) = \psi$ almost surely.
\end{lemma}
\begin{proof}
	We will first construct inverse maps locally on subsets of $M$, then use compactness to piece them together. Since $M$ is locally conformally flat, for any $p \in M$ there is a neighborhood $U$ of $p$ and an isometry $i: (U, g_\mathrm{ref}) \to (\tilde{U}, e^{2F} g_e)$, where $\tilde{U} \subset \R^n$, $F \in C^\infty (\tilde{U})$, and $g_e$ is the Euclidean metric. 
	
	We claim that without loss of generality, we can assume $F(i(p)) = 0$. Indeed, consider the map $s: \tilde{U} \to e^{F(i(p))} \tilde{U}$ given by $s(x) = e^{F(i(p))} x$. This induces an isometry from $(\tilde{U}, e^{2F} g_e)$ to $(e^{F(i(p))} \tilde{U}, s_* (e^{2F} g_e))$. The induced push-forward metric is equal to $g_e$ at the point $s(i(p))$, so if we take our isometry $i$ to be $s \circ i$ then we will have $F(i(p)) = 0$.
	
	 Since $F$ and $i$ are smooth, we can shrink $U$ to a small ball around $p$ so that $e^{nF(i(x))} \in [1-\eps, 1+\eps]$ for all $x \in U$, where $\eps < \min ( 1, \tfrac{\sqrt{2n}}{\gamma}-1)$. Since the construction of a CGF and its corresponding CGMC measure are invariant under isometries, to construct an inverse CGMC map on $(U, g_\mathrm{ref})$ it suffices to construct one on the isometric copy $(\tilde{U}, e^{2F} g_e)$.
	
	Consider a CGF on $(\tilde{U}, e^{2F} g_e)$. On $(\tilde{U}, g_e)$, the co-polyharmonic operator is just $(-\Delta)^{n/2}$ (Example 1.6 of \cite{DHKS}), so by Proposition \ref{copoly-ops}(iii), the co-polyharmonic operator on $(\tilde{U}, e^{2F} g_e)$ is $e^{-nF} (-\Delta)^{n/2}$. With Dirichlet boundary conditions, its inverse has a kernel proportional to $k(x, y) = \log (\lvert x - y \rvert^{-1}) + l(x, y)$ for some $l \in C^\infty (\tilde{U} \times \tilde{U})$; here the conformal change of metric leaves the kernel unchanged by the same argument as in Proposition 2.20 of \cite{DHKS}. By elliptic regularity, this kernel differs from the restriction of the original kernel of the co-polyharmonic operator on $M$ by another smooth function on $\tilde{U} \times \tilde{U}$. Absorbing this smooth function into $l$, we see that the covariances of $\psi$ restricted to $\tilde{U}$ have the form
	\[
		\E [(\psi, u) (\psi, v)] = \int_{\tilde{U}} \int_{\tilde{U}} u(x) v(y) (\log (\lvert x - y \rvert^{-1}) + l(x, y)) \, e^{nF(x)} dx \, e^{nF(y) }dy \;.
	\]
	where $l$ is smooth. Since $\eps \leq \tfrac{\sqrt{2n}}{\gamma} - 1$, it follows from Lemma \ref{log-GMC-generalized}, which is a slight generalization of Lemma \ref{gmc-inversion}, that if we shrink $\tilde{U}$ to a slightly smaller ball so that it is compactly contained in the original $\tilde{U}$, then there is a measurable inverse map $\tilde{X}^\gamma$ from the space of Borel measures on $\tilde{U}$ to $\CD'(\tilde{U})$ such that $\tilde{X}^\gamma (e^{\gamma \psi} e^{nF} dx) = \psi$ almost surely.
	
	Pulling everything back by the isometry $i$, we thus have an inverse map $X^\gamma$ to $M^\gamma_\mathrm{ref}$ on a neighborhood of $p$.  We can repeat this construction for any $p \in M$. By compactness, there is a finite sub-collection of these inverse maps whose respective domains cover $M$. We claim that these inverse maps can be glued together to obtain a global inverse map on all of $M$. If $p \in U_1 \subset U_2$ and we have constructed inverse maps as above on $U_1$ and $U_2$, then the outputs of the two maps agree as fields on $U_1$, i.e. they give equal outputs when paired with test functions compactly supported in $U_1$. This is clear from the construction in Lemma \ref{log-GMC-generalized}. Now if $p, q \in M$ have neighborhoods $U_p, U_q$ on which we have constructed inverse maps and $U_p \cap U_q \neq \emptyset$, the two inverses must agree on functions compactly supported in $U_p \cap U_q$. This shows that the fields obtained from these local inverse maps are compatible, and can thus be patched together via a partition of unity to obtain a global inverse $X^\gamma (M^\gamma_\mathrm{ref} (\psi))$.

\end{proof}

\subsection{Symmetrizing measure for NQF}\label{sym-nqf}

Now we will use CGMC measures to interpret the densities \eqref{nqf-measure} and \eqref{qf-measure}, which we expect to be symmetric for the corresponding stochastic dynamics. We will analyze them one at a time, starting with the NQF density \eqref{nqf-measure}. For convenience, recall the formal expression: \begin{align*}
	\omega (f)^{2 Q_\mathrm{ref} (1) / (n \sigma^2)} \exp \left( - \sigma^{-2} \omega_\mathrm{ref} (\phi P_\mathrm{ref} \phi + 2 Q_\mathrm{ref} \phi) \right) \, \omega_\mathbf{g} (d\omega)
\end{align*}
where we assume that $f \in C^\infty (M)$ is positive and that conditions (A1) and (A2) hold. Letting $\psi = \sqrt{2 / (a_n \sigma^2)} \phi$, this can be rewritten as \[
	\omega (f)^{2 Q_\mathrm{ref} (1) / (n \sigma^2)} \exp \left(- Q_\mathrm{ref} \sqrt{\frac{2a_n}{\sigma^2}} \omega_\mathrm{ref} (\psi) \right) \exp \left( -\frac{1}{2} \omega_\mathrm{ref} (\psi p_\mathrm{ref} \psi) \right) \, \omega_\mathbf{g} (d\omega) \;.
\]
We recognize the last exponential as the formal density of an ungrounded CGF. Making this identification, the expression becomes \begin{align}\label{nqf-measure-calc}
	\omega (f)^{2 Q_\mathrm{ref} (1) / (n \sigma^2)} \exp \left(- Q_\mathrm{ref} \sqrt{\frac{2a_n}{\sigma^2}} \omega_\mathrm{ref} (\psi + c) \right) \, \mu_\mathrm{ref} (d\psi) \, dc \;.
\end{align}
Since $\psi \sim \mu_\mathrm{ref}$ is grounded, $\omega_\mathrm{ref} (\psi) = 0$ almost surely. This lets us simplify to obtain \begin{align}\label{nqf-measure-calc-2}
	\omega (f)^{2 Q_\mathrm{ref} (1) / (n \sigma^2)} \exp \left(-c Q_\mathrm{ref} \sqrt{\frac{2a_n}{\sigma^2}} \omega_\mathrm{ref} (1) \right) \, \mu_\mathrm{ref} (d\psi) \, dc \;.
\end{align}
If we momentarily fix $c \in \R$ and look at the marginal distribution of $\psi$, we see that it can be normalized to a probability measure whenever \begin{align}\label{nqf-marginal}
	\E_{\mu_\mathrm{ref}} \left[ \omega (f)^{2 Q_\mathrm{ref} (1) / (n \sigma^2)} \right] < \infty \;.
\end{align}
Since $\omega = e^{n \phi} \omega_\mathrm{ref}$ but we have changed variables, we must recompute an expression for $\omega$ in terms of $\psi$ and $c$. We find that
 \[
	\omega = e^{n \phi} \omega_\mathrm{ref} = e^{\gamma c} M^\gamma_\mathrm{ref} (\psi)
\]
where \[
	\gamma = \frac{n \sqrt{a_n \sigma^2}}{\sqrt{2}} = \frac{n \sigma}{(4\pi)^{n/4} \sqrt{(n/2 - 1)!}} \;.
\]
In order for this measure to be well-defined in the sense of Proposition \ref{CGMC}, we need $\gamma < \sqrt{2n}$, or equivalently \begin{align}\label{sigma-bound}
	\sigma^2 < \frac{2 (4\pi)^{n/2} (n/2 - 1)!}{n} \;.
\end{align}
The right-hand side of this inequality grows rapidly in $n$, so this condition is most strict when $n=2$, where it becomes $\sigma^2 < 4\pi$. This precisely matches the condition found in the main result of \cite{DS22}. 

For any smooth $f$, $\omega (f)$ is bounded above in absolute value by a constant times $\omega (1)$. Therefore, moments of $M^\gamma_\mathrm{ref} (\psi) (f)$ are bounded so long as the corresponding moments of $M^\gamma_\mathrm{ref} (\psi) (1)$ are. By Proposition \ref{CGMC}(ii), for inequality \eqref{nqf-marginal} to hold we need $2Q_\mathrm{ref} (1) / (n \sigma^2) < 2n / \gamma^2$. Solving for $Q_\mathrm{ref} (1)$, this is equivalent to \[
	Q_\mathrm{ref} (1) < (4\pi)^{n/2} (n/2 - 1)! \;.
\]
Recall that in Section \ref{q-curv} we assumed that $Q_\mathrm{ref} (1) < Q_r (1)$. Let us compare that to the condition we just obtained. The $Q$-curvature of $S^n$ is the constant function $Q_r = (n-1)!$ (see \cite{Ndiaye07}), so \[
	Q_r (1) = (n-1)! \omega_r (1) = \frac{2^{n/2 + 1} \pi^{n/2} (n-1)!}{(n-1)!!} = (4\pi)^{n/2} (n/2-1)!
\]
which is exactly the same bound. In other words, integrability of the marginal always holds because we have assumed condition (A2). 

Using $\omega(f) = e^{\gamma c} M^\gamma_\mathrm{ref} (\psi)$ in \eqref{nqf-measure-calc-2} and expanding the definition of $\gamma$, we see that the density simplifies to \[
	M^\gamma_\mathrm{ref} (\psi) (f)^{2 Q_\mathrm{ref} (1) / (n \sigma^2)} \, \mu_{\mathrm{ref}} (d\psi) \, dc \;.
\]
Note that the density is entirely independent of $c$, so the measure is translation invariant. In particular, this implies that whenever the marginals for fixed $c$ have finite measure, the measure is $\sigma$-finite. We denote this measure, which can now be interpreted rigorously as a measure on $H^{0-}_\mathrm{ref}$ using \eqref{nqf-measure-calc}, by $\nunqf$. We will make use of the following integrability properties for this measure.

\begin{lemma}\label{nqf-int}
	For any $\eps \in (0, 1)$, \[
		\nunqf (\{\psi: M^\gamma_\mathrm{ref} (\psi) (1) \in (\eps, \eps^{-1}) \}) < \infty \;.
	\]
	Moreover, suppose $Z: \mathring{H}^{0-}_\mathrm{ref} \to \R$ is such that $Z(\psi)$ is a centered Gaussian when $\psi \sim \mu_\mathrm{ref}$. Then the random variable \[
	Z(\psi - \omega_\mathrm{ref} (\psi) / \omega_\mathrm{ref} (1)) \one_{\{M^\gamma_\mathrm{ref} (\psi) (1) \in (\eps, \eps^{-1}) \}}
\]
is in $L^p (\nunqf)$ for all $p \geq 1$. 
\end{lemma}
\begin{proof}
	Starting with the first claim, we compute \begin{align*}
		&\int_{H^{0-}_\mathrm{ref}} \one_{\{M^\gamma_\mathrm{ref} (\psi) (1) \in (\eps, \eps^{-1}) \}} \, \nunqf(d\psi) \\
		&= \int_\R \int_{\mathring{H}^{0-}_\mathrm{ref}} \one_{\{\psi: M^\gamma_\mathrm{ref} (\psi+c) (1) \in (\eps, \eps^{-1}) \}} M^\gamma_\mathrm{ref} (\psi) (f)^{2 Q_\mathrm{ref} (1) / (n \sigma^2)}\, \mu_\mathrm{ref} (d\psi) \, dc \\
		&\lesssim \int_\R \int_{\mathring{H}^{0-}_\mathrm{ref}}\one_{\{M^\gamma_\mathrm{ref} (\psi) (1) \in (\eps e^{-\gamma c}, \eps^{-1} e^{-\gamma c}) \}} M^\gamma_\mathrm{ref} (\psi) (1)^{2 Q_\mathrm{ref} (1) / (n \sigma^2)}\, \mu_\mathrm{ref} (d\psi) \, dc \\
		&\lesssim \sum_{k=-\infty}^\infty \int_{k \lvert \log \eps \rvert}^{(k+1) \lvert \log \eps \rvert}\int_{\mathring{H}^{0-}_\mathrm{ref}} \one_{M^\gamma_\mathrm{ref} (\psi) (1) \in (\eps^{1+\gamma(k+1)}, \eps^{-1+\gamma k})} \eps^{2 \gamma k Q_\mathrm{ref} (1) / (n \sigma^2)}\, \mu_\mathrm{ref} (d\psi) \, dc  \\
		&\lesssim \sum_{k=-\infty}^\infty \int_{\mathring{H}^{0-}_\mathrm{ref}} \one_{M^\gamma_\mathrm{ref} (\psi) (1) \in (\eps^{1+\gamma(k+1)}, \eps^{-1+\gamma k})} \eps^{2 \gamma k Q_\mathrm{ref} (1) / (n \sigma^2)}\, \mu_\mathrm{ref} (d\psi)  \\
		&\lesssim \E_{\mu_\mathrm{ref}} [M^\gamma_\mathrm{ref} (\psi) (1)^{2 Q_\mathrm{ref} (1) / (n \sigma^2)}] < \infty
	\end{align*}
	where the constants from line to line only depend on $\eps, Q_\mathrm{ref} (1), \sigma$, $n$, and $f$.
	
	Next we consider the second claim. Including $\lvert Z(\psi - \omega_\mathrm{ref} (\psi) / \omega_\mathrm{ref} (1)) \rvert^p$ in the integrand in the first line above, we can apply the same argument to obtain the upper bound \[
		\sum_{k=-\infty}^\infty \int_{\mathring{H}^{0-}_\mathrm{ref}} \lvert Z(\psi) \rvert^p \one_{M^\gamma_\mathrm{ref} (\psi) (1) \in (\eps^{1+\gamma(k+1)}, \eps^{-1+\gamma k})} \eps^{2 \gamma k Q_\mathrm{ref} (1) / (n \sigma^2)}\, \mu_\mathrm{ref} (d\psi) \;.
	\]
	Applying Hölder's inequality with some conjugate $r$ and $r^*$ in $(1, \infty)$ allows us to bound this from above by \begin{align*}
		&\sum_{k=-\infty}^\infty \| \lvert Z(\psi)\rvert^p \|_{L^{r^*} (\mu_\mathrm{ref})} \left( \int_{\mathring{H}^{0-}_\mathrm{ref}} \one_{M^\gamma_\mathrm{ref} (\psi) (1) \in (\eps^{1+\gamma(k+1)}, \eps^{-1+\gamma k})} \eps^{2 \gamma k r Q_\mathrm{ref} (1) / (n \sigma^2)}\, \mu_\mathrm{ref} (d\psi) \right)^{1/r} \\
		&\lesssim \sum_{k=-\infty}^\infty \eps^{2 \gamma k Q_\mathrm{ref} (1) / (n \sigma^2)} P_{\mu_\mathrm{ref}} ( M^\gamma_\mathrm{ref} (\psi) (1) \in (\eps^{1+\gamma(k+1)}, \eps^{-1+\gamma k}))^{1/r} \;.
	\end{align*}
	By our moment bounds on $M^\gamma_\mathrm{ref} (\psi) (1)$ together with Markov's inequality, we have \[
		P_{\mu_\mathrm{ref}} ( M^\gamma_\mathrm{ref} (\psi) (1) \leq \eps^{-1+\gamma k}) = P_{\mu_\mathrm{ref}} (( M^\gamma_\mathrm{ref} (\psi) (1))^{-\alpha} \geq \eps^{-\alpha(-1+\gamma k)}) \lesssim \eps^{\alpha \gamma k}
	\]
	for any $\alpha > 0$ and \[
		P_{\mu_\mathrm{ref}} ( M^\gamma_\mathrm{ref} (\psi) (1) \geq \eps^{1+\gamma (k+1)}) \lesssim \eps^{-\beta \gamma k}
	\]
	for any $\beta \in (0, 2n/\gamma^2)$, where the constants in the inequalities are uniform in $k$. Therefore, the sum above is bounded (up to a constant depending on $\alpha$, $\beta$, and $r$ as above) by \[
		\sum_{k=0}^\infty \eps^{2\gamma k Q_\mathrm{ref} (1) / (n \sigma^2)} \eps^{\alpha \gamma k/r} + \sum_{k=1}^{\infty} \eps^{-2\gamma k Q_\mathrm{ref} (1) / (n \sigma^2)} \eps^{\beta \gamma k/r} \;.
	\]
	If $Q_\mathrm{ref} (1) \in [0, Q_r (1))$ then the first sum is finite for any $\alpha > 0$ and the second is finite as long as $\beta$ is sufficiently close to $2n/\gamma^2$ and $r$ is sufficiently close to $1$. If $Q_\mathrm{ref} (1) < 0$, then the first sum is finite for sufficiently large $\alpha$ and the second sum is finite for any $\beta > 0$. 
\end{proof}

By Lemma \ref{CGMC-inversion}, we can equivalently consider $\nunqf$ as a measure on $H_\mathrm{ref}^{0-}$ with respect to $\psi$, or on $\CM$ with respect to $\omega$. We will abuse notation and write $\nunqf (d\psi)$ or $\nunqf (d\omega)$ in each case even though one is, strictly speaking, a pushforward of the other under an invertible measurable map. 

\subsection{Symmetrizing measure for LQF}

Next we define the LQF measure \eqref{qf-measure}, which has formal density \[
	\exp \left( - \sigma^{-2} \omega_\mathrm{ref} (\phi P_\mathrm{ref} \phi + 2\rho Q_\mathrm{ref} \phi) + 2 (n \sigma^2)^{-1} \omega (f) \right) \, \omega_\mathbf{g} (d\omega)
\]
where we assume that $f \in C^\infty (M)$ is nonpositive and that conditions (A1) and (A2') hold. This is in the family of Polyakov-Liouville measures defined in \cite{DHKS}. They show that, provided $f$ is constant and some additional constraints on the parameters hold, this measure is finite. This is an interesting special case because it extends the connection between conformal flows and Liouville quantum gravity observed in \cite{DS22} to higher dimensions. We will discuss this connection further in Section \ref{lqg}. However, the measure is still $\sigma$-finite under much more general conditions.

Proceeding as we did with the NQF measure, write the formal density as \[
	\exp \left( -\rho Q_\mathrm{ref} \sqrt{\frac{2a_n}{\sigma^2}} \omega_\mathrm{ref} (\psi) + \frac{2}{n \sigma^2} \omega (f) \right) \exp \left( -\frac{1}{2} \psi p_0 \psi \right) \, d\omega_\mathbf{g} \;.
\]
Interpreting the last exponential as the density for an ungrounded CGF, this becomes \[
	\exp \left( -\rho Q_\mathrm{ref} \sqrt{\frac{2a_n}{\sigma^2}}  \omega_\mathrm{ref} (\psi+c) + \frac{2}{n \sigma^2} e^{\gamma c} M^\gamma_\mathrm{ref} (\psi) (f) \right) \, \mu_\mathrm{ref} (d\psi) \, dc\;.
\]
Here we still require the inequality \eqref{sigma-bound} as in the previous subsection to ensure that the CGMC measure is well-defined. Once again we have $\omega_\mathrm{ref} (\psi) = 0$ almost surely when $\psi \sim \mu_\mathrm{ref}$, so this simplifies to \begin{align}\label{qf-measure-calc}
	\exp \left( -c \rho Q_\mathrm{ref} \sqrt{\frac{2a_n}{\sigma^2}}  \omega_\mathrm{ref} (1) + \frac{2}{n \sigma^2} e^{\gamma c} M^\gamma_\mathrm{ref} (\psi) (f) \right) \, \mu_\mathrm{ref} (d\psi) \, dc\;.
\end{align}

Denote this measure by $\nuqf$. Unlike for NQF, we do not expect the marginals to be finite when $c$ is fixed. However, we still have integrability properties analogous to Lemma \ref{nqf-int}.

\begin{lemma}\label{qf-int}
	For any $\eps \in (0, 1)$, \[
		\nuqf (\{\psi: M^\gamma_\mathrm{ref} (\psi) (1) \in (\eps, \eps^{-1}) \}) < \infty \;.
	\]
	In particular, $\nuqf$ is $\sigma$-finite. Moreover, suppose $Z: \mathring{H}^{0-}_\mathrm{ref} \to \R$ is such that $Z(\psi)$ is a centered Gaussian when $\psi \sim \mu_\mathrm{ref}$. Then the random variable \[
	Z(\psi - \omega_\mathrm{ref} (\psi) / \omega_\mathrm{ref} (1)) \one_{\{M^\gamma_\mathrm{ref} (\psi) (1) \in (\eps, \eps^{-1}) \}}
\]
is in $L^p (\nuqf)$ for all $p \geq 1$. 
\end{lemma}
\begin{proof}
	Let us first bound the following integral: \begin{align*}
		&\int_\R \int_{\mathring{H}^{0-}_\mathrm{ref}} \exp \left( -c \rho Q_\mathrm{ref} \sqrt{\frac{2a_n}{\sigma^2}} \omega_\mathrm{ref} (1) \right) \one_{M^\gamma_\mathrm{ref} (\psi+c) (1) \in (\eps, \eps^{-1})} \, \mu_\mathrm{ref} (d\psi) \, dc\\
		&\lesssim \int_\R \int_{\mathring{H}^{0-}_\mathrm{ref}} \exp \left( -c \rho Q_\mathrm{ref} \sqrt{\frac{2a_n}{\sigma^2}} \omega_\mathrm{ref} (1) \right) \one_{\{M^\gamma_\mathrm{ref} (\psi) (1) \in (\eps e^{-\gamma c}, \eps^{-1} e^{-\gamma c}) \}} \, \mu_\mathrm{ref} (d\psi) \, dc \\
		&\lesssim \sum_{k=-\infty}^\infty \int_{k \lvert \log \eps \rvert}^{(k+1) \lvert \log \eps \rvert}\int_{\mathring{H}^{0-}_\mathrm{ref}} \eps^{k \rho Q_\mathrm{ref} \sqrt{2a_n / \sigma^2} \omega_\mathrm{ref} (1)} \one_{M^\gamma_\mathrm{ref} (\psi) (1) \in (\eps^{1+\gamma(k+1)}, \eps^{-1+\gamma k})} \, \mu_\mathrm{ref} (d\psi) \, dc  \\
		&\lesssim \sum_{k=-\infty}^\infty \int_{\mathring{H}^{0-}_\mathrm{ref}} \eps^{2 \rho \gamma k Q_\mathrm{ref}(1) / (n \sigma^2)}  \one_{M^\gamma_\mathrm{ref} (\psi) (1) \in (\eps^{1+\gamma(k+1)}, \eps^{-1+\gamma k})} \, \mu_\mathrm{ref} (d\psi)  \\
		&\lesssim \E_{\mu_\mathrm{ref}} [M^\gamma_\mathrm{ref} (\psi) (1)^{2 \rho Q_\mathrm{ref} (1) / (n \sigma^2)}]
	\end{align*}
	which is finite whenever condition (A2') holds, by Proposition \ref{CGMC}(ii). 
	
	To compute the $\nuqf$ measure of the set in the lemma, we still need to include the exponential factor $\exp (2e^{\gamma c} M^\gamma_\mathrm{ref} (\psi) (f) / ( n \sigma^2))$ in the integrand. Since $f$ is assumed non-positive for LQF, this exponential lies in $(0, 1]$ pointwise. The integrand was otherwise non-negative, so including this term does not affect the convergence of the integral and we can conclude the first claim of the lemma.
	
	The second claim follows in exactly the same way the analogous claim was proved in Lemma \ref{nqf-int}. Indeed, a similar application of Hölder's inequality yields that the desired $L^p$ norm is bounded up to a constant by \[
		\sum_{k=-\infty}^\infty \eps^{2 \rho \gamma k Q_\mathrm{ref} (1) / (n \sigma^2)} P_{\mu_\mathrm{ref}} ( M^\gamma_\mathrm{ref} (\psi) (1) \in (\eps^{1+\gamma(k+1)}, \eps^{-1+\gamma k}))^{1/r} \;.
	\]
	This sum can be controlled for all $Q_\mathrm{ref} (1) \in (-\infty, \rho^{-1} Q_r (1))$ using Markov's inequality in exactly the same way as before.
\end{proof}

Let us summarize the results of the last two subsections. We now have explicit meanings for the measures $\nunqf$ and $\nuqf$ associated to NQF and LQF respectively. Provided that inequality \eqref{sigma-bound} and condition (A2) hold, the NQF measure $\nunqf$ is $\sigma$-finite for any smooth $f > 0$. If instead $f \leq 0$ and inequality \eqref{sigma-bound} and condition (A2') hold, then $\nuqf$ is $\sigma$-finite. 

\section{Integration by parts}

The standard Dirichlet inner product for smooth compactly-supported functions on $\R^n$ has the integration-by-parts formula \[
	\int_{\R^n} DF \cdot DG \, dx = \int_{\R^n} F (-\Delta G) \, dx \;.
\]
To construct weak solutions to NQF and LQF, we will make use of bilinear forms $\CE (F, G)$ defined similarly to the left-hand side above, but using our newly constructed measures $\nunqf$ and $\nuqf$. To do computations with such a bilinear form, it will be convenient to rewrite it as an integral of $F (-\CL G)$ for some operator $\CL$ which plays the role of $\Delta$. In this section we will derive integration-by-parts formulas for $\nunqf$ and $\nuqf$ which allow us to compute the corresponding operators $\CL$. 

Suppose $\psi$ is a CGF with law $\mu_\mathrm{ref}$. Since $\psi$ has covariance operator $p_\mathrm{ref}^{-1}$, the Cameron-Martin directions in $\mathring{H}^{0-}_\mathrm{ref}$ are given by $p_\mathrm{ref}^{-1/2} (\mathring{H}^{0-}_\mathrm{ref}) = \mathring{H}_\mathrm{ref}^{(n/2)-}$. With no loss we can work with the slightly smaller space $\mathring{E} \coloneq \mathring{H}_\mathrm{ref}^{n/2}$, equipped with inner product \[
	\langle h_1, h_2 \rangle_E \coloneq \langle \sqrt{p_\mathrm{ref}} h_1, \sqrt{p_\mathrm{ref}} h_2 \rangle_{L^2 (\omega_\mathrm{ref})} = \langle p_\mathrm{ref} h_1, h_2 \rangle_{L^2 (\omega_\mathrm{ref})} \;.
\]
This inner product extends to the ungrounded space $E \coloneq H_{\mathrm{ref}}^{n/2}$ by letting $\langle 1, h \rangle_E = 0$ for all $h$, which makes sense because $p_\mathrm{ref} 1 = 0$. 

We will need dense subclasses of functionals on $H^{0-}_\mathrm{ref}$ for which we can prove our integration-by-parts formulas. The following classes were used to a similar effect in \cite{DS22}: 
\begin{definition}
	Denote by $\tilde{\CC}$ the space of functionals on $H^{0-}_\mathrm{ref}$ of the form \begin{align}\label{test-functional}
		G(\psi) = q(M^\gamma_\mathrm{ref} (\psi) (h_0), \dots, M^\gamma_\mathrm{ref} (\psi) (h_k))
	\end{align}
	where $q \in C^2 (\R^{k+1})$ and $h_i \in C^\infty (M)$, with $h_0$ equal to the constant function $1$. Let $\CC \subset \tilde{\CC}$ be the subset of functionals where $q$ can be chosen with support contained in $(\eps, \eps^{-1}) \times Q$ for some $\eps > 0$ and compact $Q \subset \R^k$.
\end{definition}

We start by computing Fréchet derivatives of these functionals in Cameron-Martin directions.

\begin{lemma}\label{frechet}
	Suppose $G \in \tilde{\CC}$ is of the form \eqref{test-functional} and $h \in E$ is continuous. Then for $\mu_\mathrm{ref}$-a.e. $\psi$, \[
		D_h G (\psi) = \gamma \sum_{i=0}^k \partial_i q (M^\gamma_\mathrm{ref} (\psi) (h_0), \dots, M^\gamma_\mathrm{ref} (\psi) (h_k)) M^\gamma_\mathrm{ref} (\psi) (h_i h) \;.	
	\]
	Moreover, if $G \in \CC$ then for all continuous $h \in E$ there is a constant $C$ depending only on $G$ and $h$ such that $ \lvert D_h G (\psi) \rvert \leq C$ for all $\psi \in \mathring{H}^{0-}_\mathrm{ref}$.
\end{lemma}
\begin{proof}
	By part (i) of Proposition \ref{CGMC}, \[
		M^\gamma_\mathrm{ref} (\psi + th) = e^{t \gamma h} M^\gamma_\mathrm{ref} (\psi)
	\]
	for $\mu_\mathrm{ref}$-a.e. $\psi$. Using this, we can compute discrete differences of $G$:
	\begin{align*}
		&G(\psi + th) - G(\psi) \\
		&= q((e^{t \gamma h} M^\gamma_\mathrm{ref} (\psi)) (h_0), \dots, (e^{t \gamma h} M^\gamma_\mathrm{ref} (\psi)) (h_k)) \\
		&\qquad- q(M^\gamma_\mathrm{ref} (\psi) (h_0), \dots, M^\gamma_\mathrm{ref} (\psi) (h_k))  \\
		&= \sum_{i=0}^k \bigg( q(M^\gamma_\mathrm{ref} (\psi)(h_0), \dots, M^\gamma_\mathrm{ref} (\psi) (h_{i-1}), \\
		&\qquad\qquad (e^{t \gamma h} M^\gamma_\mathrm{ref} (\psi)) (h_i), \dots, (e^{t \gamma h} M^\gamma_\mathrm{ref} (\psi)) (h_k))\\
		&\qquad - q(M^\gamma_\mathrm{ref} (\psi)(h_0), \dots, M^\gamma_\mathrm{ref} (\psi) (h_{i}), \\
		&\qquad\qquad (e^{t \gamma h} M^\gamma_\mathrm{ref} (\psi)) (h_{i+1}), \dots, (e^{t \gamma h} M^\gamma_\mathrm{ref} (\psi)) (h_k)) \bigg) \\
		&= t \sum_{i=0}^k \partial_i q (M^\gamma_\mathrm{ref} (\psi) (h_0), \dots, (e^{t'_i \gamma h} M^\gamma_\mathrm{ref} (\psi))(h_i), \dots, (e^{t \gamma h} M^\gamma_\mathrm{ref} (\psi)) (h_k)) \\
		&\qquad \cdot \frac{d}{ds}\bigg\vert_{s=t'_i} (e^{s \gamma h} M^\gamma_\mathrm{ref} (\psi)) (h_i)
	\end{align*}
	where the last equality uses the mean-value theorem with $t'_i \in [0, t]$ for each $i$. We can rewrite the derivative term as \begin{align*}
		\frac{d}{ds} \bigg\vert_{s=t'_i} \int_{\mathring{H}_\mathrm{ref}^{0-}} e^{s \gamma h} h_i \, M^\gamma_\mathrm{ref} (\psi)
	\end{align*}
	which equals $\gamma M^\gamma_\mathrm{ref} (\psi) (e^{t'_i \gamma h} h_i h)$ by dominated convergence. Dividing the expression for $G(\psi+th) - G(\psi)$ by $t$ and taking a limit as $t$ approaches zero, we recover the desired formula for $D_h G$. 
	
	For the last claim, suppose $G \in \CC$. By continuity of $h_i$ and $h$, there is a constant $C$ depending only on $G$ and $h$ such that $\lvert M^\gamma_\mathrm{ref} (\psi) (h_i h) \rvert \leq C M^\gamma_\mathrm{ref} (\psi) (1)$ for all $i$. If $M^\gamma_\mathrm{ref} (\psi) (1)$ is outside of $(\eps, \eps^{-1})$ then all of the partial derivatives of $q$ are zero by the definition of $\CC$. The partial derivatives of $q$ are uniformly bounded above by another constant $C'$, so we conclude from the formula for $D_h G (\psi)$ that \[
		\lvert D_h G (\psi) \rvert \leq \gamma (k+1) C' C \eps^{-1}
	\]
	for all $\psi \in \mathring{H}^{0-}_\mathrm{ref}$. 
\end{proof}

Before proving integration-by-parts formulas for $\nunqf$ and $\nuqf$, we start with similar formulas for grounded and ungrounded CGFs.

\begin{lemma}\label{grounded-cgf-ibp}
	For all continuous $h \in E$ and $G \in \CC$, \begin{equation}\label{grounded-cgf-ibp-formula}
		\int_{\mathring{H}^{0-}_\mathrm{ref}} D_h G(\psi) - D_{\overline{h}} G(\psi) \, \mu_\mathrm{ref} (d\psi) = \int_{\mathring{H}^{0-}_\mathrm{ref}} G(\psi) \langle h, \psi \rangle_E \, \mu_\mathrm{ref} (d\psi)
	\end{equation}
	where $\overline{h}$ is the constant function $\omega_\mathrm{ref} (h) / \omega_\mathrm{ref} (1)$. 
\end{lemma}

\begin{remark}
	In \eqref{grounded-cgf-ibp-formula} we wrote $\langle h, \psi \rangle_E$ even though $\psi$ may not be in $E$.  Instead, inner products of the form $\langle h, \psi \rangle_E$ for $h \in E$ are defined using the Paley-Wiener map. For details of this construction see Section 1.7 of \cite{DP06}.
\end{remark}

\begin{proof}
	Since both sides are linear in $h$ and $E = \mathring{E} \oplus \R$, it suffices to first show that the formula holds if $h \in \mathring{E}$, then show it holds when $h=1$. For $h \in \mathring{E}$, we have by the Cameron-Martin formula that \begin{align*}
		&\int_{\mathring{H}^{0-}_\mathrm{ref}} G(\psi+th) \, \mu_\mathrm{ref} (d\psi) \, dc =\int_{\mathring{H}^{0-}_\mathrm{ref}} G(\psi) \exp \left( -\frac{t^2}{2} \| h \|^2_E + t \langle h, \psi \rangle_E \right) \, \mu_\mathrm{ref} (d\psi) \, dc \;.
	\end{align*}
	Viewing both sides as functions of $t$, we would like to differentiate at $t=0$. The derivative of the integrand on the left-hand side is bounded by the previous lemma, so we can swap the derivative and the integral by the Leibniz integral rule: \begin{align*}
		&\frac{d}{dt} \left[ \int_{\mathring{H}^{0-}_\mathrm{ref}} G(\psi + th) \, \mu_\mathrm{ref} (d\psi) \right] \bigg\vert_{t=0} \\
		&= \int_{\mathring{H}^{0-}_\mathrm{ref}} \frac{d}{dt} \left[ G(\psi + th) \right] \bigg\vert_{t=0} \mu_\mathrm{ref} (d\psi) \\
		&= \int_{\mathring{H}^{0-}_\mathrm{ref}}  D_h G (\psi)) \, \mu_\mathrm{ref} (d\psi) \;.
	\end{align*}
	
	 To apply the same argument to the right-hand side, note that $G(\psi)$ is bounded and the exponential term satisfies \[
		\left\lvert \frac{d}{dt} \exp \left( -\frac{t^2}{2} \| h\|^2_E + t \langle h, \psi \rangle_E \right) \right\rvert \leq C \exp (C' \lvert \langle h, \psi \rangle_E \rvert )
	\]
	for some constants $C, C' > 0$, uniformly over $t$ in a small interval around zero. This upper bound is integrable with respect to $\mu_\mathrm{ref}$ because $\langle h, \psi \rangle_E$ is Gaussian. Therefore, we can apply the Leibniz integral rule to the right-hand side as well, so we obtain 
		\begin{align*}
		&\frac{d}{dt} \left[ \int_{\mathring{H}^{0-}_\mathrm{ref}}G(\psi ) \exp \left( -\frac{t^2}{2} \|h\|^2_E + t \langle h, \psi \rangle_E \right) \, \mu_\mathrm{ref} (d\psi) \right] \bigg\vert_{t=0} \\
		&= \int_{\mathring{H}^{0-}_\mathrm{ref}} \frac{d}{dt} \left[G(\psi) \exp \left( -\frac{t^2}{2} \|h\|^2_E + t \langle h, \psi \rangle_E \right) \right] \bigg\vert_{t=0} \mu_\mathrm{ref} (d\psi)\\
		&= \int_{\mathring{H}^{0-}_\mathrm{ref}} G (\psi) \langle h, \psi \rangle_E \, \mu_\mathrm{ref} (d\psi) \;.
	\end{align*}
	Equating these two derivatives yields \eqref{grounded-cgf-ibp-formula} in the case $h \in \mathring{E}$.
	
	Next suppose $h = 1$. Then $h = \overline{h}$ and $\langle h, \psi \rangle_E = 0$, so \eqref{grounded-cgf-ibp-formula} still holds. 
\end{proof}

From this, the next result for the ungrounded CGF follows quickly.

\begin{lemma}\label{cgf-ibp}
	For all continuous $h \in E$, $G \in \CC$, and $r \in C^\infty (\R)$, \begin{equation}\label{cgf-ibp-formula}
		\int_{H^{0-}_\mathrm{ref}}  D_h (r(\omega_\mathrm{ref} (\psi)) G (\psi)) \, \tilde{\mu}_\mathrm{ref} (d\psi) = \int_{H^{0-}_\mathrm{ref}} r(\omega_\mathrm{ref} (\psi)) G(\psi) \langle h, \psi \rangle_E \, \tilde{\mu}_\mathrm{ref} (d\psi)
\end{equation}
whenever both integrals converge.
\end{lemma}

\begin{proof}
	As in the previous lemma, suppose first that $h \in \mathring{E}$. Then \[
		D_h (r(\omega_\mathrm{ref} (\psi)) G (\psi)) = r(\omega_\mathrm{ref} (\psi)) D_h G(\psi)
	\]
	by the Leibniz rule. This means that the left-hand side of the formula in the lemma is \[
		\int_\R r(c \omega_\mathrm{ref} (1)) \int_{\mathring{H}^{0-}_\mathrm{ref}} D_h G(\psi+c) \, \mu_\mathrm{ref} (d\psi) \, dc
	\]
	and the right-hand side is \[
		\int_\R r(c \omega_\mathrm{ref} (1)) \int_{\mathring{H}^{0-}_\mathrm{ref}} G(\psi+c) \langle h, \psi\rangle_E \, \mu_\mathrm{ref} (d\psi) \, dc \;.
	\]
	Since $G(\cdot + c)$ can be written as a functional in $\CC$, the previous lemma allows us to equate the inner integrals for fixed $c$, so the outer integrals are equal whenever they converge.
	
	Next suppose $h = 1$. Provided they converge, both sides of the formula in the lemma are zero by translation invariance of $\tilde{\mu}_\mathrm{ref}$, so equality still holds.
\end{proof}

We will now use this to prove similar formulas for $\nunqf$ and $\nuqf$, which will be central to the proof of Theorem \ref{main-theorem}.

\begin{theorem}[Integration-by-parts for NQF]\label{nqf-ibp}
	For all continuous $h \in E$ and $G \in \CC$,
	 \begin{align*}
		\int_{H^{0-}_\mathrm{ref}} G(\psi) \langle h, \psi \rangle_E \, \nunqf (d\psi) = &\int_{H^{0-}_\mathrm{ref}} D_h G (\psi) + \frac{2 \gamma Q_\mathrm{ref} (1)}{n \sigma^2} G(\psi) \frac{M^\gamma_\mathrm{ref} (\psi) (fh)}{M^\gamma_\mathrm{ref} (\psi) (f)} \\
		&- \frac{2\gamma Q_\mathrm{ref} \omega_\mathrm{ref} (h)}{n \sigma^2} G(\psi)  \, \nunqf (d\psi)
	\end{align*}
\end{theorem}
\begin{proof}
	Let us first check that both sides are integrable. Up to a constant multiple, we can bound $G(\psi)$ and $D_h G$ by $\one_{M^\gamma_\mathrm{ref} (\psi) (1) \in [\eps, \eps^{-1}]}$. This gives integrability of all the terms on the right-hand side. For the left-hand side, we need to show that $ \langle h, \psi \rangle_E \one_{M^\gamma_\mathrm{ref} (\psi) (1) \in [\eps, \eps^{-1}]}$ is integrable with respect to $\nunqf$. In fact, since $\langle h, \psi \rangle_E$ is Gaussian we know this is in $L^p (\nunqf)$ for all $p \geq 1$ by Lemma \ref{nqf-int}.
	
	By the definition of $\nunqf$, the left-hand side equals \[
		\int_{H^{0-}_\mathrm{ref}} G(\psi) \langle h, \psi \rangle_E  (M^\gamma_\mathrm{ref} (\psi - \frac{\omega_\mathrm{ref} (\psi)}{\omega_\mathrm{ref} (1)})(f))^{2Q_\mathrm{ref} (1)) / (n \sigma^2)} \, \tilde{\mu}_\mathrm{ref} (d\psi)		
	\]
	which is of the form needed for Lemma \ref{cgf-ibp} with functional \[
		F(\psi) = G(\psi) M^\gamma_\mathrm{ref} (\psi) (f)^{2Q_\mathrm{ref} (1) / (n \sigma^2)}
	\]
	and \[
		r(x) = e^{-2 \gamma Q_\mathrm{ref} x / (n \sigma^2)} \;.
	\]
	Note that we are using the identity $Q_\mathrm{ref} (1) / \omega_\mathrm{ref} (1) = Q_\mathrm{ref}$ to deduce $r$. At first glance it seems that $F$ is only in $\tilde{\CC}$ and not $\CC$ because the exponential is not compactly supported. However, if $M^\gamma_\mathrm{ref} (\psi) (f)$ is sufficiently large then $M^\gamma_\mathrm{ref} (\psi) (1)$ will be large enough to make $G(\psi)$, and hence $F(\psi)$, zero. Thus, $F$ is only nonzero when $M^\gamma_\mathrm{ref} (\psi) (f)$ is in a compact interval, and so $F$ can be written as a functional in $\CC$.
	
	Using Lemma \ref{frechet} and the Leibniz rule, we can compute the derivative \begin{align*}
		&D_h (r(\omega_\mathrm{ref} (\psi)) F(\psi)) = \\
&D_h G (\psi) e^{-2 \gamma Q_\mathrm{ref} \omega_\mathrm{ref} (\psi) / (n \sigma^2)} M^\gamma_\mathrm{ref} (\psi) (f)^{2Q_\mathrm{ref} (1) / (n \sigma^2)} \\
		&+ \frac{2 \gamma Q_\mathrm{ref} (1)}{n \sigma^2} G (\psi) e^{-2 \gamma Q_\mathrm{ref} \omega_\mathrm{ref} (\psi) / (n \sigma^2)} M^\gamma_\mathrm{ref} (\psi) (f)^{2Q_\mathrm{ref} (1) / (n \sigma^2) - 1} M^\gamma_\mathrm{ref} (\psi) (fh) \\
		&- \frac{2\gamma Q_\mathrm{ref}}{n \sigma^2} G (\psi) e^{-2 \gamma Q_\mathrm{ref} \omega_\mathrm{ref} (\psi) / (n \sigma^2)} M^\gamma_\mathrm{ref} (\psi) (f)^{2Q_\mathrm{ref} (1) / (n \sigma^2)} \omega_\mathrm{ref} (h) \;.
	\end{align*}
	Applying Lemma \ref{cgf-ibp}, we thus obtain the integral \begin{align*}
		&\int_{H^{0-}_\mathrm{ref}} D_h G (\psi) e^{-2 \gamma Q_\mathrm{ref} \omega_\mathrm{ref} (\psi) / (n \sigma^2)} M^\gamma_\mathrm{ref} (\psi) (f)^{2Q_\mathrm{ref} (1) / (n \sigma^2)} \\
		&+ \frac{2 \gamma Q_\mathrm{ref} (1)}{n \sigma^2} G (\psi) e^{-2 \gamma Q_\mathrm{ref} \omega_\mathrm{ref} (\psi) / (n \sigma^2)} M^\gamma_\mathrm{ref} (\psi) (f)^{2Q_\mathrm{ref} (1) / (n \sigma^2) - 1} M^\gamma_\mathrm{ref} (\psi) (fh) \\
		&- \frac{2\gamma Q_\mathrm{ref}}{n \sigma^2} G (\psi) e^{-2 \gamma Q_\mathrm{ref} \omega_\mathrm{ref} (\psi) / (n \sigma^2)} M^\gamma_\mathrm{ref} (\psi) (f)^{2Q_\mathrm{ref} (1) / (n \sigma^2)} \omega_\mathrm{ref} (h) \, \tilde{\mu}_\mathrm{ref} (d\psi)
	\end{align*}
	Reabsorbing the density of $\nunqf$ from the integrand, this becomes \begin{align*}
		\int_{H^{0-}_\mathrm{ref}} D_h G (\psi) + \frac{2 \gamma Q_\mathrm{ref} (1)}{n \sigma^2} G (\psi) \frac{M^\gamma_\mathrm{ref} (\psi) (fh)}{M^\gamma_\mathrm{ref} (\psi) (f)} - \frac{2\gamma Q_\mathrm{ref}}{n \sigma^2} G (\psi) \omega_\mathrm{ref} (h) \, \nunqf (d\psi)
	\end{align*}
	as desired.
\end{proof}

\begin{theorem}[Integration-by-parts for LQF]\label{qf-ibp}
	For all continuous $h \in E$ and $G \in \CC$, \begin{align*}
		\int_{H_\mathrm{ref}^{0-}} G(\psi) \langle h, \psi \rangle_E \, \nuqf(d\psi) = &\int_{H_\mathrm{ref}^{0-}} D_h G (\psi) + \frac{2 \gamma}{n \sigma^2} G(\psi) M^\gamma_\mathrm{ref} (\psi) (fh) \\
		&- \rho Q_\mathrm{ref} \sqrt{2a_n / \sigma^2} G(\psi) \omega_\mathrm{ref} (h) \, \nuqf (d\psi) \;.
	\end{align*}
\end{theorem}
\begin{proof}
	The proof is very similar to that of the previous result. In particular, the integrability of both sides follows from an analogous argument using Lemma \ref{qf-int} instead of Lemma \ref{nqf-int}.
	
	By the definition of $\nuqf$, the left-hand side equals \begin{equation}\label{qf-ibp-calculation}
		\int_{H^{0-}_\mathrm{ref}} e^{-\rho Q_\mathrm{ref} \sqrt{2a_n / \sigma^2} \omega_\mathrm{ref}(\psi)} \exp \left(\frac{2}{n \sigma^2} M^\gamma_\mathrm{ref} (\psi) (f) \right) G(\psi) \langle h, \psi \rangle_E \, \tilde{\mu}_\mathrm{ref} (d\psi) \;.
	\end{equation}
	Consider the functional \[
		F(\psi) = G(\psi) \exp \left(\frac{2}{n \sigma^2} M^\gamma_\mathrm{ref} (\psi) (f) \right) \;.
	\]
	As before, we note that if $M^\gamma_\mathrm{ref} (f)$ is too large in absolute value then $M^\gamma_\mathrm{ref} (1)$ will be large enough that $G(\psi)$ is zero. Thus, $F$ lies in $\CC$. Applying Lemma \ref{frechet} with $r(x) = e^{-\rho Q_\mathrm{ref} \sqrt{2a_n / \sigma^2} x}$, we use the Leibniz rule to find \begin{align*}
		&D_h (r(\omega_\mathrm{ref} (\psi)) F (\psi)) =\\
		&D_h G (\psi) e^{-\rho Q_\mathrm{ref} \sqrt{2a_n / \sigma^2} \omega_\mathrm{ref} (\psi)} \exp \left(\frac{2}{n \sigma^2} M^\gamma_\mathrm{ref} (\psi) (f) \right) \\
		&+ \gamma \frac{2}{n \sigma^2} G(\psi) e^{-\rho Q_\mathrm{ref} \sqrt{2a_n / \sigma^2} \omega_\mathrm{ref} (\psi)} \exp \left(\frac{2}{n \sigma^2} M^\gamma_\mathrm{ref} (\psi) (f) \right) M^\gamma_\mathrm{ref} (\psi) (fh) \\
		&- \rho Q_\mathrm{ref} \sqrt{2a_n / \sigma^2} G(\psi) e^{-\rho Q_\mathrm{ref} \sqrt{2a_n / \sigma^2} \omega_\mathrm{ref} (\psi)} \exp \left(\frac{2}{n \sigma^2} M^\gamma_\mathrm{ref} (\psi) (f) \right) \omega_\mathrm{ref} (h) \;.
	\end{align*}
	Applying Lemma \ref{cgf-ibp} to \eqref{qf-ibp-calculation} yields \begin{align*}
		&\int_{H^{0-}_\mathrm{ref}} D_h G (\psi) e^{-\rho Q_\mathrm{ref} \sqrt{2a_n / \sigma^2} \omega_\mathrm{ref} (\psi)} \exp \left(\frac{2}{n \sigma^2} M^\gamma_\mathrm{ref} (\psi) (f) \right) \\
		&+ \gamma \frac{2}{n \sigma^2} G(\psi) e^{-\rho Q_\mathrm{ref} \sqrt{2a_n / \sigma^2} \omega_\mathrm{ref} (\psi)} \exp \left(\frac{2}{n \sigma^2} M^\gamma_\mathrm{ref} (\psi) (f) \right) M^\gamma_\mathrm{ref} (\psi) (fh) \\
		&- \rho Q_\mathrm{ref} \sqrt{2a_n / \sigma^2} G(\psi) e^{-\rho Q_\mathrm{ref} \sqrt{2a_n / \sigma^2} \omega_\mathrm{ref} (\psi)} \exp \left(\frac{2}{n \sigma^2} M^\gamma_\mathrm{ref} (\psi) (f) \right) \omega_\mathrm{ref} (h) \, \tilde{\mu}_\mathrm{ref} (d\psi) \\
		=& \int_{H^{0-}_\mathrm{ref}} G(\psi) e^{-\rho Q_\mathrm{ref} \sqrt{2a_n / \sigma^2} \omega_\mathrm{ref}(\psi)} \exp \left(\frac{2}{n \sigma^2} M^\gamma_\mathrm{ref} (\psi) (f) \right) \langle h, \psi \rangle_E \, \tilde{\mu}_\mathrm{ref} (d\psi) \;.
	\end{align*}
	As before, turning these into integrals with respect to $\nuqf$ by absorbing the density $d\nuqf / d\tilde{\mu}_\mathrm{ref}$ on both sides finishes the proof.
\end{proof}

\section{Dirichlet form analysis}\label{dirichlet}

In this section we prove Theorem \ref{main-theorem}. The main idea will be to exploit the well-known correspondence between Dirichlet forms and symmetric Markov processes. For convenience, we start by recalling some elements of this correspondence and explaining how they will be used in the proof. We generally follow the notation of \cite{FOT11}, where many more details can be found.

Let $X$ be a locally compact separable metric space and let $m$ be a positive $\sigma$-finite Radon measure on $X$. A non-negative symmetric bilinear form $\CE$ on $L^2 (X, m)$ with dense domain $\CD [\CE]$ is simply referred to as a symmetric form. A sequence $(u_n)_{n \geq 1}$ in $\CD [\CE]$ is $\CE$-Cauchy if $\CE(u_n - u_m, u_n - u_m) \to 0$ as $n, m \to \infty$, and it $\CE$-converges to $u \in \CD[\CE]$ if $\CE (u_n - u, u_n - u) \to 0$ as $n \to \infty$.

\begin{definition}
	Let $\CE$ be a symmetric form. 
	\begin{enumerate}[(i)]
		\item $\CE$ is \emph{closed} if every sequence in $\CD [\CE]$ which is both Cauchy (in $L^2$) and $\CE$-Cauchy also $\CE$-converges to an element of $\CD [\CE]$. It is \emph{closable} if it has a closed extension.
		\item For $\eps > 0$, an \emph{$\eps$-Markovian function} is an increasing $1$-Lipschitz function $\tau_\eps: \R \to \R$ such that $\tau_\eps (t) = t$ for $t \in [0, 1]$ and $\tau_\eps (t) \in [-\eps, 1+\eps]$ for all $t \in \R$.
		\item $\CE$ is \emph{Markovian} if for all $\eps > 0$ there is an $\eps$-Markovian function $\tau_\eps$ such that for all $u \in \CD [\CE]$, $\tau_\eps (u) \in \CD [\CE]$ with $\CE (\tau_\eps (u), \tau_\eps (u)) \leq \CE (u, u)$.
	\end{enumerate}
	If $\CE$ is both closed and Markovian, it is called a \emph{Dirichlet form}. 
\end{definition}

Some results in \cite{FOT11} also assume that $\supp m = X$. However, this makes no difference for us in practice because any Dirichlet form on $L^2 (X, m)$ can be identified with one on the isometric space $L^2 (\supp m, m)$ which inherits all of its relevant properties. With this in mind, we omit this assumption without loss of generality as in \cite{DS22}. The reason why we care about Dirichlet forms in this context is the following correspondence, which combines several results from Chapter 1 of \cite{FOT11}.

\begin{proposition}
	There is a one-to-one correspondence between closed symmetric forms $\CE$ and non-positive self-adjoint operators $A$ on $L^2 (X, m)$, where an operator $A$ corresponds to the form $\CE (u, v) = (\sqrt{-A} u, \sqrt{-A} v)$. In this correspondence, $\CE$ is Markovian (and hence a Dirichlet form) if and only if the strongly continuous semigroup $(T_t)_{t \geq 0}$ generated by $A$ is Markovian. 
\end{proposition}
In the above, $(T_t)_{t \geq 0}$ is Markovian if for all $t \geq 0$, $T_t u \in [0, 1]$ $m$-almost everywhere whenever $u \in [0, 1]$ $m$-almost everywhere.

Now consider a Markov process on $(X, \CB (X))$ with semigroup $(T_t)_{t \geq 0}$ and generator $A$. $A$ is non-positive, and if it is also self-adjoint then this proposition allows us to associate a Dirichlet form $\CE$ to the process. A necessary and sufficient condition for $A$ to be self-adjoint is that the process is $m$-symmetric, i.e. for all $t \geq 0$ and $u, v$ in the domain of $T_t$, \[
	\int_X u(x) (T_t v)(x) \, m(dx) = \int_X (T_t u) (x) v(x) \, m(dx) \;.
\]
Therefore, one can obtain a Dirichlet form from an $m$-symmetric Markov process and vice versa. It turns out that if the Dirichlet form satisfies certain additional conditions, the Markov process does as well. To fully explain this we need a few more definitions. Denote by $C_c (X)$ the space of compactly supported continuous functions on $X$ with the uniform norm.

\begin{definition}
	Let $\CE$ be a Dirichlet form on $L^2 (X, m)$. 
	\begin{enumerate}[(i)]
		\item A \emph{core} $\mathscr{C}$ of $\CE$ is a subset of $\CD [\CE] \cap C_c (X)$ which is dense in $\CD [\CE]$ with respect to the norm $\|u\|^2 = \|u\|_{L^2 (m)}^2 + \CE (u, u)$ and dense in $C_c (X)$ with respect to the uniform norm. $\CE$ is \emph{regular} if it admits a core.
		\item A core $\mathscr{C}$ is \emph{standard} if it is a linear subspace of $C_c (X)$ and for every $\eps > 0$, there is an $\eps$-Markovian function $\tau_\eps$ such that $u \in \mathscr{C}$ implies $\tau_\eps (u) \in \mathscr{C}$. It is \emph{special standard} if it is a subalgebra of $C_c (X)$ and, for every $K \subseteq U \subseteq X$ with $K$ compact and $U$ relatively compact and open, there is a non-negative $u$ in $\mathscr{C}$ such that $u=1$ on $K$ and $u=0$ on $X \setminus U$. 
		\item $\CE$ is \emph{local} if whenever $u, v \in \CD [\CE]$ have disjoint compact supports, $\CE (u, v) = 0$. It is \emph{strongly local} if whenever $u, v \in \CD [\CE]$ have compact supports and $v$ is constant on a neighborhood of the support of $u$, $\CE (u, v) = 0$. 
	\end{enumerate}
\end{definition}

These properties holding for $\CE$ are sufficient for the existence of an $m$-symmetric Markov process with desirable properties, as the next proposition details. We say that the Markov process is killed when it first hits the cemetery state $\delta$.

\begin{proposition}[\cite{FOT11}, Chapters 4 and 7]
	If $\CE$ is a regular Dirichlet form on $L^2 (X, m)$, then there is an $m$-symmetric Hunt process on $(X, \CB (X))$ with associated Dirichlet form $\CE$. Moreover:
	\begin{enumerate}[(i)]
		\item If $\CE$ is local, then this Hunt process is a diffusion. This means that for quasi-every starting point in $X$, the process is almost surely continuous up to the (possibly infinite) time it is killed.
		\item If $\CE$ is strongly local, then this diffusion is not killed inside $X$ for quasi-every starting point in $X$. In other words, if $\zeta$ is the random time at which the process is killed, then almost surely either $\zeta = \infty$ or the left limit of the process at $\zeta$ is not in $X$.
	\end{enumerate}
\end{proposition}

\begin{remark}
	In the above proposition we used the term ``quasi-every" point in $X$. This is a notion of largeness related to the Dirichlet form $\CE$. A precise definition can be found in \cite{FOT11}; here we simply note that quasi-every implies $m$-almost every. One could strengthen Definition \ref{weak-sol} by replacing ``almost every $z \in \CM$" with ``quasi-every $z \in \CM$" and Theorem \ref{main-theorem} would still hold. For simplicity, we will ignore the distinction between these two terms and stick to the ``$m$-almost every" terminology.
	
	Also note that in Definition \ref{weak-sol}, we can now interpret ``almost every $z \in \CM$" to mean ``$\nunqf$-almost every" for NQF and ``$\nuqf$-almost every" for LQF. 
\end{remark}

With this correspondence in mind, one can imagine how the proof of Theorem \ref{main-theorem} will proceed. We will first construct Dirichlet forms on $L^2 (\CM, \nunqf)$ and $L^2 (\CM, \nuqf)$ associated to NQF and LQF, then show that they satisfy the relevant properties so that the corresponding Hunt processes are symmetric diffusions. It will then remain to show that these diffusions are actually weak solutions.

To analyze these processes, we will use a few more facts about the correspondence between Hunt processes and Dirichlet forms. In the next proposition we refer to certain processes as additive functionals. This is a certain class of processes which we will cite results about, but will not work with directly and so do not need a precise definition for. We refer the reader to Appendix A.2 and Chapter 5 of \cite{FOT11} for background on these processes.

\begin{proposition}[\cite{FOT11}, Chapter 5]\label{revuz}
	Let $(\Omega, \CF, (z_t)_{t \in [0, \infty]}, (P_z)_{z \in X})$ be a Hunt process associated to a regular Dirichlet form $\CE$, and suppose $u \in D[\CE]$. \begin{enumerate}[(i)]
	\item The process \[
		A^{[u]}_t = u(z_t) - u(z_0)
	\]
	is a continuous additive functional which decomposes uniquely as \[
		A^{[u]} = S^{[u]} + N^{[u]}
	\]
	where $S^{[u]}$ is a finite-energy martingale additive functional and $N^{[u]}$ is a zero-energy continuous additive functional. Here the energy of an additive functional $(A_t)_{t \geq 0}$ is given by \[
		e(A) = \lim_{t \to 0} \frac{1}{2t} \E_m (A_t^2) \;.
	\]
	\item To each positive continuous additive functional $A$, there is an associated measure $\mu_A$, called its \emph{Revuz measure}. If $A^1$ and $A^2$ have the same Revuz measure, then they are equivalent in the sense that for each $t > 0$, they are almost surely equal for $m$-almost every initial condition.
	\item The Revuz measure of the quadratic variation of $S^{[u]}$ satisfies \[
		\mu_{\langle S^{[u]} \rangle} (v) = 2 \CE (uv, u) - \CE (u^2, v) \;.
	\]
\end{enumerate}
\end{proposition}

The Revuz correspondence will allow us to find explicit formulas for one-dimensional projections of our Hunt processes. Using the SDEs \eqref{nqf-sde} and \eqref{qf-sde}, we can find similar formulas that weak solutions to NQF and LQF must obey. Showing that these are the same will prove Theorem \ref{main-theorem}.

This general proof strategy is essentially the same one used by \cite{DS22} for the two-dimensional case, and several of the steps mentioned above will carry over from their setting to the present one with minimal modification. We will indicate when this is the case.

\subsection{Construction of the Dirichlet form}

Recall that $\CM$ is a locally compact separable metric space, and that $\nunqf$ and $\nuqf$ are $\sigma$-finite Radon measures on $\CM$ by Lemmas \ref{nqf-int} and \ref{qf-int}. As explained above, we will ultimately want to define our Dirichlet forms on the spaces $L^2 (\CM, \nunqf)$ and $L^2 (\CM, \nuqf)$. It suffices to first construct forms on $L^2 (H^{0-}_\mathrm{ref}, \nunqf)$ and $L^2 (H^{0-}_\mathrm{ref}, \nuqf)$, then push them forward under the appropriate CGMC map.

First, we will obtain an expression for the gradient of a functional $G \in \CC$. Since $M^\gamma_\mathrm{ref} (\psi)$ is a Radon measure, the space of continuous $h \in E$ is dense in $L^2 (M^\gamma_\mathrm{ref} (\psi))$ (this follows by Urysohn's lemma, as in Lemma 2.3 of \cite{DS22}). Therefore, Lemma \ref{frechet} and the Riesz representation theorem guarantee the existence of an $L^2 (M^\gamma_\mathrm{ref} (\psi))$-gradient of $G$ with formula \[
	DG(\psi) = \gamma \sum_{i=0}^k \partial_i q (M^\gamma_\mathrm{ref} (\psi) (h_0), \dots, M^\gamma_\mathrm{ref} (\psi) (h_k)) h_i
\]
where $G$ has the same structure as in Definition \ref{test-functional}.

\begin{definition}\label{dirichlet-form}
	Let $\CE_\mathrm{NQF}$ and $\CE_\mathrm{LQF}$ be bilinear forms on $\CC$ which for $F, G \in \CC$ are given by \[
	\CE_\mathrm{NQF} (F, G) = \frac{n^2 \sigma^2}{2\gamma^2} \int_{H^{0-}_\mathrm{ref}} \langle DF (\psi), DG (\psi) \rangle_{L^2 (M^\gamma_\mathrm{ref} (\psi))} \, \nu_\mathrm{NQF} (d\psi)
	\]
and	
	\[
		\CE_\mathrm{LQF} (F, G) = \frac{n^2 \sigma^2}{2\gamma^2} \int_{H^{0-}_\mathrm{ref}} \langle DF (\psi), DG (\psi) \rangle_{L^2 (M^\gamma_\mathrm{ref} (\psi))} \, \nu_\mathrm{LQF} (d\psi) \;.
	\]
\end{definition}

From the definition it is clear that these forms are symmetric and positive semidefinite. To prove any further properties, we will need a less symmetric expression for these forms. For this we will use our integration-by-parts formulas. 

First consider the NQF measure $\nunqf$. Let \[
	F(\psi) = p(M^\gamma_\mathrm{ref} (\psi) (f_0), \dots, M^\gamma_\mathrm{ref} (\psi) (f_m))
	\]
	and 
	\[
	G(\psi) = q(M^\gamma_\mathrm{ref} (\psi) (g_0), \dots, M^\gamma_\mathrm{ref} (\psi) (g_l))
	\]
	be functionals in $\CC$. For the rest of this section, the functions $p$ and $q$ will sometimes be written without arguments; in these cases they will always be assumed to have the same arguments as above. From Theorem \ref{nqf-ibp} we compute \begin{align*}
		&\sum_{i=0}^l \int_{H^{0-}_\mathrm{ref}} p \partial_i q \langle g_i, \psi \rangle_E \, \nunqf (d\psi) \\
		&=\sum_{i=0}^l \int_{H^{0-}_\mathrm{ref}} D_{g_i} (p \partial_i q ) + \frac{2 \gamma Q_{\mathrm{ref}} (1)}{n \sigma^2} p \partial_i q \frac{M^\gamma_\mathrm{ref} (\psi) (fg_i)}{M^\gamma_\mathrm{ref} (\psi) (f)} - \frac{2 \gamma Q_\mathrm{ref} \omega_\mathrm{ref} (g_i)}{n \sigma^2} (p \partial_i q) \, \nunqf (d\psi) \\
		&= \sum_{i=0}^l \int_{H^{0-}_\mathrm{ref}} \gamma \partial_i q \left( \sum_{j=0}^m (\partial_j p) M^\gamma_\mathrm{ref} (\psi) (f_j g_i) \right) + \gamma p \left( \sum_{j=0}^l (\partial_i \partial_j q) M^\gamma_\mathrm{ref} (\psi) (g_j g_i) \right)\\
		&\qquad + \frac{2 \gamma Q_{\mathrm{ref}} (1)}{n \sigma^2} p \partial_i q \frac{M^\gamma_\mathrm{ref} (\psi) (fg_i)}{M^\gamma_\mathrm{ref} (\psi) (f)}- \frac{2 \gamma Q_\mathrm{ref} \omega_\mathrm{ref} (g_i)}{n \sigma^2} (p \partial_i q) \, \nunqf (d\psi) \;.	
\end{align*}

The term involving the first sum over $j$ is 
\begin{align*}
	&\sum_{i=0}^l \int_{H^{0-}_\mathrm{ref}} \gamma \partial_i q \left( \sum_{j=0}^m (\partial_j p) M^\gamma_\mathrm{ref} (\psi) (f_j g_i) \right) \, \nunqf (d\psi)\\
	&= \frac{1}{\gamma} \int_{H^{0-}_\mathrm{ref}} \langle DF (\psi), DG (\psi) \rangle_{L^2(M^\gamma_\mathrm{ref} (\psi))} \, \nunqf (d\psi) \\
	&= \frac{2\gamma}{n^2 \sigma^2} \CE_\mathrm{NQF} (F, G) \;.
\end{align*}
Aside from this one, the rest of the terms are integrals against $p$ with respect to $\nunqf$. Therefore, we can define an operator $\CL_\mathrm{NQF}$ to cancel all of these terms. Rearranging our earlier computation to solve for $\CE_\mathrm{NQF} (F, G)$ and using the definitions of $\gamma$ and $a_n$ gives the following lemma. 
\begin{lemma}
	Define an operator $\CL_\mathrm{NQF}$ on $\CC$ by \begin{align*}
		\CL_\mathrm{NQF} G(\psi) = &\sum_{i=0}^l (\partial_i q) \left( \frac{n Q_\mathrm{ref} (1) M^\gamma_\mathrm{ref} (\psi) (fg_i)}{M^\gamma_\mathrm{ref} (\psi) (f)}\right) + \frac{n^2 \sigma^2}{2} \sum_{i,j=0}^l (\partial_i \partial_j q) M^\gamma_\mathrm{ref} (\psi) (g_j g_i) \\
		&- n Q_\mathrm{ref} \sum_{i=0}^l (\partial_i q) \omega_\mathrm{ref} (g_i) - \frac{n^2 \sigma^2}{2 \gamma} \sum_{i=0}^l \partial_i q\langle g_i, \psi \rangle_E
	\end{align*}
	with $G$ as above. Then \[
		\CE_\mathrm{NQF} (F, G) = \int_{H^{0-}_\mathrm{ref}} F(\psi) (-\CL_\mathrm{NQF} G) (\psi) \, \nunqf (d\psi) 
	\]
	for all $F, G \in \CC$. 
\end{lemma}

Let $\CC^\CM$ be the set of functionals on $\CM$ of the form $F(\omega) = q(\omega(f_0), \dots, \omega(f_k))$ where $q \in C^2 (\R^{k+1})$, $f_i \in C^\infty (M)$, $f_0$ is the constant function $1$, and $q$ has compact support contained in $(\eps, \eps^{-1}) \times Q$ for some $\eps > 0$ and compact $Q \subset \R^k$ (recall Definition \ref{test-functional}). For any such $F$, the functional $F \circ M^\gamma_\mathrm{ref}$ belongs to $\CC$. We can define $\CE^{\CM}_\mathrm{NQF}$ on $\CC^\CM$ via \[
	\CE^\CM_\mathrm{NQF} (F, G) \coloneq \CE_\mathrm{NQF} (F \circ M^\gamma_\mathrm{ref}, G \circ M^\gamma_\mathrm{ref}) \;.
\]
This form is also symmetric and non-negative. 

We repeat a similar computation for LQF. By Theorem \ref{qf-ibp}, \begin{align*}
		&\sum_{i=0}^l \int_{H^{0-}_\mathrm{ref}} p \partial_i q \langle g_i, \psi \rangle_E \, \nuqf (d\psi) \\
		&=\sum_{i=0}^l \int_{H^{0-}_\mathrm{ref}} D_{g_i} (p \partial_i q ) + \frac{2 \gamma}{n \sigma^2} (p \partial_i q) M^\gamma_\mathrm{ref} (\psi) (fg_i) \\
		&\qquad - \rho Q_\mathrm{ref} \sqrt{2 a_n / \sigma^2} (p \partial_i q) \omega_\mathrm{ref} (g_i) \, \nuqf (d\psi) \\
		&= \sum_{i=0}^l \int_{H^{0-}_\mathrm{ref}} \gamma \partial_i q \sum_{j=0}^m \partial_j p M^\gamma_\mathrm{ref} (\psi) (f_j g_i) + \gamma p \sum_{j=0}^l (\partial_i \partial_j q) M^\gamma_\mathrm{ref} (\psi) (g_j g_i)\\
		&\qquad + \frac{2 \gamma}{n \sigma^2} (p \partial_i q) M^\gamma_\mathrm{ref} (\psi) (fg_i) - \rho Q_\mathrm{ref} \sqrt{2 a_n / \sigma^2} (p \partial_i q) \omega_\mathrm{ref} (g_i) \nuqf (d\psi) \;.	
\end{align*}
Just like NQF, the term involving the first sum over $j$ is $\tfrac{2\gamma}{n^2 \sigma^2}\CE_\mathrm{LQF} (F, G)$. Thus, we can solve for $\CE_\mathrm{LQF} (F, G)$ and define $\CL_\mathrm{LQF}$ so that it encompasses all the other terms:

\begin{lemma}
	Define $\CL_\mathrm{LQF}$ on $\CC$ by \begin{align*}
		\CL_\mathrm{LQF} G(\psi) &= \sum_{i=0}^l n (\partial_i q) M^\gamma_\mathrm{ref} (\psi) (fg_i) + \frac{n^2 \sigma^2}{2} \sum_{i,j=0}^l (\partial_i \partial_j q) M^\gamma_\mathrm{ref} (\psi) (g_j g_i) \\
		&- \frac{n^2 \sigma^2}{2\gamma} \sum_{i=0}^l \partial_i q \langle g_i, \psi \rangle_E  - n \rho Q_\mathrm{ref} \sum_{i=0}^l (\partial_i q) \omega_\mathrm{ref} (g_i)\;.
	\end{align*}
	Then \[
		\CE_\mathrm{LQF} (F, G) = \int_{H^{0-}_\mathrm{ref}} F(\psi) (-\CL_\mathrm{LQF} G (\psi)) \, \nuqf (d\psi) 
	\]
	for all $F, G \in \CC$. 
\end{lemma}

Again we can obtain from this a form $\CE^\CM_\mathrm{LQF}$ on $\CC^\CM$ by letting $\CE^\CM_\mathrm{LQF} (F, G) = \CE_\mathrm{LQF} (F \circ M^\gamma_\mathrm{ref}, G \circ M^\gamma_\mathrm{ref})$. Using the $\CL$ operators it will be much easier to prove that $\CE^\CM_\mathrm{NQF}$ and $\CE^\CM_\mathrm{LQF}$ are closable. Based on the results at the beginning of this section, we will want to show that these are Dirichlet forms, that they are regular with a special standard core, and that they are strongly local. 

\begin{lemma}
	The forms $\CE^\CM_\mathrm{NQF}$ and $\CE^\CM_\mathrm{LQF}$ are closable. We denote their minimal closed extensions also by $\CE^\CM_\mathrm{NQF}$ and $\CE^\CM_\mathrm{LQF}$. 
\end{lemma}
\begin{proof}
Let $s \in \{\mathrm{NQF}, \mathrm{LQF}\}$. To show that $\CE^\CM_s$ is closable, it suffices to check that whenever a sequence $(F_n)_{n \geq 1}$ in $\CC^\CM$ converges to $0$ in $L^2 (\CM, \nu_s)$, then $\CE^\CM_s (F_n, G) \to 0$ for all $G \in \CC^\CM$ (see Exercise 1.1.2 of \cite{FOT11}). Choose a sequence $(F_n)_{n \geq 1}$ in $\CC^\CM$ which converges to zero in $L^2 (\CM, \nu_s)$, as well as an arbitrary $G \in \CC^\CM$. We abuse notation and also write $F_n$ and $G$ for $F_n \circ M^\gamma_\mathrm{ref}$ and $G \circ M^\gamma_\mathrm{ref}$, the corresponding elements of $\CC$. Then we have \begin{align*}
	\lvert \CE^\CM_s (F_n, G) \rvert &= \left\lvert \int_{\CM} F_n (\omega)(-\CL_s G) (\omega) \, \nu_s (d\omega) \right\rvert \\
	&\leq \| F_n \|_{L^2 (\CM, \nu_s)} \|\CL_s G \|_{L^2 (\CM, \nu_s)} \;.
\end{align*}
Thus, it suffices to show that $\CL_s G$ is in $L^2 (\CM, \nu_s)$. Inspecting the definitions of $\CL_\mathrm{NQF}$ and $\CL_\mathrm{LQF}$, we note that $q$ and all of its derivatives are bounded above up to constants by $\one_{M^\gamma_\mathrm{ref} (\psi) \in (\eps, \eps^{-1})}$. It is then immediate from Lemmas \ref{nqf-int} and \ref{qf-int} that each term of $\CL_\mathrm{NQF} G$ and $\CL_\mathrm{LQF} G$ is in $L^2 (\nunqf)$ or $L^2 (\nuqf)$ as needed.
\end{proof}

Propositions 3.9 and 3.10 of \cite{DS22} establish that $\CC^\CM$ is a special standard core. Their argument is in the case where $M$ is the two-dimensional torus, but it extends with no loss to our current setting. Moreover, they show that the $\eps$-Markovian functions $\tau_\eps$ required for a standard core can be taken to be smooth.

To show Markovianity, locality, and strong locality of the minimal closed extensions, it suffices to prove the analogous properties for the original forms; see Section 3.1 of \cite{FOT11}. From this, Markovianity of the forms follow rather quickly as in \cite{DS22}. Indeed, for $F \in \CC^\CM$ of the form $F(\omega) = p(\omega (f_0), \dots, \omega (f_k))$, we have \[
	D(\tau_\eps \circ F) (\omega) = \gamma \sum_{i=0}^k (\tau_\eps' \circ p) (\omega(f_0), \dots, \omega(f_k)) \partial_i p (\omega(f_0), \dots, \omega(f_k)) f_i \;.
\]
Therefore, $	\CE^\CM_s (\tau_\eps \circ F, \tau_\eps \circ F) \leq \CE^\CM_s (F, F)$ for each $s \in \{\mathrm{NQF}, \mathrm{LQF}\}$ by the definition of $\CE^\CM_s$ and the fact that $\tau_\eps' (t) \in [0, 1]$ for all $t \in \R$. 

The proof of strong locality in Proposition 3.10 of \cite{DS22} also applies directly to our setting, as it only uses the inner product structure of $\CE$ in Definition \ref{dirichlet-form} together with a derivative formula of the type proved in Lemma \ref{frechet}. Thus, both forms are strong local. By the Dirichlet form correspondence, we can conclude the following: 
\begin{proposition}
	There is an $\nunqf$ (resp. $\nuqf$)-symmetric Hunt process on $\CM$ associated to the form $\CE^\CM_\mathrm{NQF}$ (resp. $\CE^\CM_\mathrm{LQF}$). For $\nunqf$ (resp. $\nuqf$)-almost every starting point, the process has continuous paths and is not killed inside $\CM$. 
\end{proposition}

\subsection{Analysis of the Hunt processes}\label{fukushima}

The remaining task is to show that the Hunt processes we have constructed are indeed weak solutions to the NQF and LQF equations. Let us first focus on the SDE \eqref{nqf-sde} associated to weak solutions for NQF: \begin{align*}
	\partial_t \omega_t (h) = &-n \left( Q_\mathrm{ref} \omega_\mathrm{ref} (h) + \omega_\mathrm{ref} (h P_\mathrm{ref} \phi_t) - \frac{Q_t (1)}{\omega_t (f)} \omega_t (fh) \right) + n \sigma \omega_t (\xi_t h) \;.
\end{align*}
This can be used to obtain an equation for a functional in $\CC^\CM$ of the form \[
	F(\omega) = p(\omega (f_0), \dots, \omega (f_m)) \;.
\]
Indeed, deriving $F_t \coloneq F(\omega_t)$ and using Itô's formula gives \begin{align*}
	\partial_t F_t = &\Bigg( n \sum_{i=0}^m (\partial_i p) \left( \frac{Q_t (1)}{\omega_t (f)} \omega_t (ff_i) - Q_\mathrm{ref} \omega_\mathrm{ref} (f_i) - \omega_\mathrm{ref} (f_i P_\mathrm{ref} \phi_t) \right) \\
	&+ \frac{n^2 \sigma^2}{2} \sum_{i,j=0}^m (\partial_i \partial_j p) \omega_t (f_i f_j) \Bigg)
	+ n \sigma 	\sum_{i=0}^m (\partial_i p) \omega_t (\xi_t f_i)
\end{align*}
where the second derivative term arises from the white noise isometry. On the other hand, recall that \begin{align*}
	\CL_\mathrm{NQF} F (\omega_t) = &n \sum_{i=0}^m (\partial_i p) \Bigg( \frac{Q_t (1) \omega_t (f f_i)}{\omega_t (f)}  - Q_\mathrm{ref} \omega_\mathrm{ref} (f_i) - \frac{n \sigma^2}{2\gamma} \langle f_i, \psi_t \rangle_E \Bigg) \\
	&+ \frac{n^2 \sigma^2}{2} \sum_{i,j=0}^m (\partial_i \partial_j p) \omega_t (f_i f_j) \;.
\end{align*}
From the definition of $\gamma$ and the inner product on $E$, \[
\omega_\mathrm{ref} (f_i P_\mathrm{ref} \phi_t ) = \frac{n \sigma^2}{2 \gamma} \langle f_i, \psi_t \rangle_E
\]
so we can write \[
	\partial_t F_t = \CL_\mathrm{NQF} F(\omega_t)  + n \sigma \sum_{i=0}^m (\partial_i p) \omega_t (\xi_t f_i) \;.
\]
Consequently, the process \[
	S^{F}_{\mathrm{NQF}} (t) \coloneq F (\omega_t) - F(\omega_0) - \int_0^t \CL_\mathrm{NQF} F(\omega_r) \, dr
\]
is a martingale with quadratic variation \[
	\langle S^{F}_{\mathrm{NQF}} \rangle_t = n^2 \sigma^2 \sum_{i,j=0}^m \int_0^t (\partial_i p) (\partial_j p) \omega_r (f_i f_j) \, dr \;.
\]
The last two equations characterize the processes $F_t$ via their semimartingale decompositions, and consequently characterize projections $\omega_t (h)$ of weak solutions by a standard localization argument. Therefore, we will be done if we show that one-dimensional projections of the Hunt process constructed in the previous subsection have the same semimartingale decomposition. 

Before showing these properties, let us see the analogous decomposition we will need to prove for the Hunt process associated to LQF. The equation for $F$ in that case is \begin{align*}
	\partial_t F_t = &\Bigg( n \sum_{i=0}^m (\partial_i p) \left( \omega_t (f f_i) - \rho Q_\mathrm{ref} \omega_\mathrm{ref} (f_i) - \omega_\mathrm{ref} (\phi_t P_\mathrm{ref} f_i) \right)  \\
	&+ \frac{n^2 \sigma^2}{2} \sum_{i,j=0}^m (\partial_i \partial_j p) \omega_t (f_j f_i) \Bigg)
	+ n \sigma \sum_{i=0}^m (\partial_i p) \omega_t (\xi_t f_i)
\end{align*}
and \begin{align*}
	\CL_\mathrm{LQF} F (\omega_t) = &n\sum_{i=0}^m (\partial_i p) \left( \omega_t (f f_i) - \rho Q_\mathrm{ref} \omega_\mathrm{ref} (f_i)  - \frac{n \sigma^2}{2 \gamma} \langle f_i, \psi \rangle_E  \right) \\
	&+ \frac{n^2 \sigma^2}{2} \sum_{i,j=0}^m (\partial_i \partial_j p) \omega_t (f_i f_j) \;.
\end{align*}
Therefore, the process \[
	S^{F}_\mathrm{LQF} (t) = F(\omega_t) - F(\omega_0) - \int_0^t \CL_\mathrm{LQF} F(\omega_r) \, dr
\]
is a martingale with the same quadratic variation, \[
	\langle S^{F}_\mathrm{LQF} \rangle_t = n^2 \sigma^2 \sum_{i,j=0}^m \int_0^t (\partial_i p) (\partial_j p)\omega_r (f_i f_j) \, dr \;.
\]

The remainder of the argument is identical for NQF and LQF. Let $Y^\mathrm{NQF}$ and $Y^\mathrm{LQF}$ denote the Hunt processes constructed in the previous section, and consider the processes \[
	A^{[F]}_s (t) = F(Y^s_t) - F(Y^s_0)
\]
where $s \in \{ \mathrm{NQF}, \mathrm{LQF}\}$. 

Let $A^{[F]}_s (t) = S^{[F]}_s (t) + N^{[F]}_s (t)$ be the decomposition of these processes given in Proposition \ref{revuz}(i). By Proposition \ref{revuz}(iii), the Revuz measure of the quadratic variation of $S^{[F]}_s$ satisfies $\mu _{\langle S^{[F]}_s \rangle} (G) = 2\CE^\CM_s (FG, F) - \CE^\CM_s (F^2, G)$. Using the definition of $\CE^\CM_s$ and applying the Leibniz rule to the right-hand side, we compute \[
	d\mu _{\langle S^{[F]}_s \rangle} (\omega) = \frac{n^2 \sigma^2}{\gamma^2} \| DF(\omega) \|_{L^2 (\omega)}^2 \, d\nu_s (\omega) \;.
\]

On the other hand, by standard properties of the Revuz correspondence (see Lemma 3.12 of \cite{DS22} and the references therein), the additive functional \[
	t \mapsto \frac{n^2 \sigma^2}{\gamma^2} \int_0^t \| DF(Y^s_r) \|^2_{L^2 (Y^s_r)} \, dr
\]
has exactly the same Revuz measure. By Proposition \ref{revuz}(ii), we conclude: 
\begin{lemma}
	For $F \in \CC^\CM$, \[
		\langle S^{[F]}_s \rangle_t = \frac{n^2 \sigma^2}{\gamma^2} \int_0^t \| DF(Y^s_r) \|^2_{L^2 (Y^s_r)} \, dr \;.
	\]
	By Lemma \ref{frechet}, it follows that \[
		\langle S^{[F]}_s \rangle_t =  n^2 \sigma^2 \int_0^t \sum_{i,j=0}^m (\partial_i p)( \partial_j p) Y^s_r (f_i f_j) \, dr \;.
	\]
\end{lemma}
This last expression exactly matches the quadratic variation of $S^F_s$.

For the zero-energy term $N^{[F]}_s$, we can argue as in Lemma 3.14 of \cite{DS22} and Example 5.1.1 of \cite{FOT11} to obtain \[
	N^{[F]}_s (t) = \int_0^t \CL_s F(Y^s_r ) \, dr
\]
which matches the drift of $F_t$. More precisely, using integration-by-parts and Corollary 5.4.1 of \cite{FOT11}, the Revuz measure of the left-hand side is $\CL_s F (\omega) d\nu_s (\omega)$. The right-hand side has the same Revuz measure by the same standard properties used for the quadratic variation term, so the two are equivalent. Therefore, the one-dimensional projections $A^{[F]}_s$ have the same semimartingale decomposition as they would if $Y^s$ were a weak solution. The same can be said of the pairings $Y^s_t (h)$ by a localization argument, so these pairings solve the SDEs \eqref{nqf-sde} and \eqref{qf-sde}. This shows that the Hunt processes $Y^s$ are weak solutions, which finishes the proof of Theorem \ref{main-theorem}.

\section{Applications and discussion}

\subsection{Volume dynamic and invariant measure}

In this subsection we will prove Corollaries \ref{volume-corollary} and \ref{invariant-measure}. Corollary \ref{volume-corollary} follows directly from Theorem \ref{main-theorem} and the definition of a weak solution. Indeed, if one takes $h=1$ in \eqref{nqf-sde} and \eqref{qf-sde} and uses the fact that $Q_t (1) = Q_\mathrm{ref} (1)$ for all $t \geq 0$, the volume dynamics stated in the corollary appear readily.

For Corollary \ref{invariant-measure}, we will need a slightly finer analysis of the volume dynamic for LQF. We know that $\nuqf$ is a symmetrizing measure for the weak solution to LQF. As long as the process $Y^\mathrm{LQF}$ is conservative (meaning it is almost surely not killed in finite time), then $\nuqf$ is also an invariant measure. Since $Y^\mathrm{LQF}$ is a diffusion which is continuous up to the time it is killed, the only way it can leave $\CM$ in finite time is if $V_t$ shrinks to $0$ or blows up to infinity. Thus, it suffices to show that in the setting of Corollary \ref{invariant-measure}, the volume almost surely does not hit $0$ or infinity in finite time.

\begin{lemma}
	Suppose $Q_\mathrm{ref} < 0$ and $\sigma^2 \leq -2Q_\mathrm{ref} (1) / n$. Then solutions to the SDE \[
		dV_t = -n Q_\mathrm{ref} (1) \left( \rho - \frac{V_t}{V_\mathrm{ref}}\right) \, dt + n \sigma \sqrt{V_t} \, dB_t
	\]
	remain in $(0, \infty)$ for all $t > 0$ almost surely.
\end{lemma} 

\begin{proof}
	This SDE describes a CIR process with parameters $a=-n Q_\mathrm{ref}$, $b=\rho V_\mathrm{ref}$, and $n \sigma$. By standard results on CIR processes (see \cite{JYC09}), this process will almost surely not blow up to infinity, and it will almost surely not hit zero as long as $2 (-n Q_\mathrm{ref}) (\rho V_\mathrm{ref}) \geq n^2 \sigma^2$. This inequality is implied by our assumption. 
\end{proof}

\subsection{Liouville quantum gravity}\label{lqg}

Liouville conformal field theory (LCFT), also known as Liouville quantum gravity (LQG), is a canonical family of random Riemannian metrics which has been rigorously constructed for two-dimensional surfaces in \cite{DKRV16}, \cite{DRV16}, and \cite{GRV19}. A key feature of the result of \cite{DS22} is that under certain additional conditions, the flow constructed there has an LQG measure as an invariant measure. In other words, the dynamic they analyze is a stochastic quantization of LQG in the sense of Parisi and Wu (\cite{ParisiWu}).

More recently, \cite{DHKS} constructed analogous measures on a class of even-dimensional manifolds. In our notation, they consider adjusted Polyakov-Liouville measures of the form 
\[
	\nu (d\psi) = \exp (-\theta Q_\mathrm{ref} \omega_\mathrm{ref} (\psi) - m M^\gamma_\mathrm{ref} (\psi) (1)) \tilde{\mu}_\mathrm{ref} (d\psi) \;.
\]
They show that these measures are finite so long as (A1) and (A2) are satisfied, $\gamma \in (0, \sqrt{2n})$, and $\theta Q_\mathrm{ref} < 0$. Moreover, in the particular case where $Q_\mathrm{ref} < 0$ and $\theta = a_n (\tfrac{n}{\gamma} + \tfrac{\gamma}{2})$, this measure is conformally quasi-invariant (see Theorem 6.12 of \cite{DHKS} for the precise quasi-conformal behavior). Analyzing the LQF density \eqref{qf-measure-calc}, we see that this value of $\theta$ corresponds to $\rho = 1+a_n n \sigma^2 / 4$. Therefore, Corollary \ref{invariant-measure} implies that LQF with parameters $Q_\mathrm{ref} (1) < 0$, $f = Q_\mathrm{ref}$, $\rho = 1+a_n n \sigma^2 / 4$, and $\sigma^2 < -2Q_\mathrm{ref} (1) / n$ is a stochastic quantization of one of these higher-dimensional Polyakov-Liouville measures. While \cite{DHKS} constructs these measures in the regime $\gamma \in (0, \sqrt{2n})$, we only have invariant measures in the case where \[
\gamma \in \left( 0, \sqrt{2n} \wedge \sqrt{-n a_n Q_\mathrm{ref} (1) } \right)
\] 
because of the requirement that $\sigma^2 \leq -2Q_\mathrm{ref} (1) / n$ in Corollary \ref{invariant-measure}.

\subsection{Topological conditions}\label{topology}

Recall that for Theorem \ref{main-theorem} to hold, we need a closed and locally conformally flat manifold $(M, g)$ of even dimension $n$ such that (A1) and one of (A2) or (A2') holds, where these conditions are: \begin{enumerate}
	\item[(A1)] $P_g$ is positive semi-definite with kernel given by the constant functions.
	\item[(A2)] $Q_g (1) < Q_r (1)$, where $g_r$ is the round metric on the sphere $S^n$.
	\item[(A2')] $Q_g (1) < \rho^{-1} Q_r (1)$, where $g_r$ is the round metric on the sphere $S^n$.
\end{enumerate}
Now we will demonstrate that these conditions are satisfied by a wide class of manifolds. 

Let us start with condition (A1), which is also discussed in Section 2.1 of \cite{DHKS}. There, they observe that one setting where (A1) is satisfied is when $M$ is Einstein with non-negative Ricci curvature (recall that a manifold is Einstein if its Ricci curvature tensor is a scalar multiple of its metric). They also discuss a more general condition based on the spectral gap of $\Delta_g$, which can be used to show that certain hyperbolic manifolds satisfy (A1). For manifolds with positive $Q$-curvature, more is known. \cite{Gursky99} showed that in the four-dimensional case, it is sufficient that $Q_g (1) > 0$ and that $(M, g)$ has positive Yamabe invariant.

Next let us discuss condition (A2) and its stronger variant (A2'). \cite{CY95} computed $Q$-curvature explicitly for a number of notable four-dimensional examples. Some manifolds that satisfy (A2) include $S^2 \times S^2$, $S^1 \times S^3$, and the complex projective space $\mathbb{C} P^2$. For condition (A2'), the most interesting situation is that of Section \ref{lqg} when $\rho = 1 + a_n n \sigma^2 / 4$. For a fixed value of $\sigma^2$, $\rho$ rapidly approaches $1$ as the dimension increases, so this condition becomes closer and closer to (A2). Even in low dimensions, one can verify that all the manifolds from the previous paragraph satisfy (A2') as long as $\sigma^2$ isn't too large.

Lastly, we require that $M$ is locally conformally flat. Although this is a somewhat strong condition, we remark that it can likely be weakened. We used it only in Lemma \ref{CGMC-inversion} to obtain a field which could be input into Lemma \ref{log-GMC-generalized}. Since this lemma only requires that the remainder of the covariance after subtracting the logarithmic singularity is continuous and in $H^{n+\eps}_\mathrm{loc}$, rather than smooth, there may be a weaker geometric condition than local conformal flatness which guarantees this.

Let us also briefly discuss when the constant $Q$-curvature metric $g_\mathrm{ref}$ is unique. \cite{Vetois24} established uniqueness (up to scaling) of constant $Q$-curvature metrics for closed Einstein manifolds with positive scalar curvature in dimension $4$, so long as they are not conformally equivalent to the sphere. While this is not directly related to our analysis, the uniqueness of these metrics may be useful in analyzing the convergence of the stochastic flows to their invariant measures.

\begin{appendix}
\section{GMC inversion}

Recall that in the proof of Lemma \ref{CGMC-inversion}, we needed the existence of an inverse to a GMC map with a reference measure $\tilde{\omega}_{\mathrm{ref}}$ which is not Lebesgue, but does have a smooth Lebesgue density that is close to $1$. In the case of Lebesgue measure, such an inverse map was constructed in arbitrary dimension in \cite{Vihko24}. In this appendix we briefly summarize the main result and method of proof of \cite{Vihko24}, then generalize the argument to mildly non-Lebesgue measures as described above. For consistency with \cite{Vihko24} we adopt the notation used there, so it will differ slightly from the notation in the rest of this paper.

Let $D_0 \subseteq \R^d$ be a bounded domain and let $D$ be compactly contained in $D_0$. Let $G$ be a log-correlated field on $D_0$ with covariance \[
	\E [\langle G, f_1 \rangle \langle G, f_2 \rangle] = \int_{D_0 \times D_0} f_1 (x) f_2 (y) C_G (x, y) \, dx \, dy
\]
where $C_G (x, y) = \log (\lvert x - y \rvert^{-1} ) + g_G (x, y)$ and $g_G \in H^{d+\eps}_\mathrm{loc} (D_0 \times D_0)$ is continuous on the interior of $D_0$. By Theorem 
A of \cite{JSW19}, $G$ decomposes as $S+H$ on $D$, where $S$ and $H$ are centered Gaussian fields, $S$ is $\star$-scale invariant, and $H$ is Hölder continuous. More precisely, $S$ is a log-correlated field with covariance kernel \[
	C_S (x, y) = \int_1^\infty k(t(x-y)) \frac{dt}{t} = \log (\lvert x - y \rvert^{-1}) + g_S (x, y)
\]
where $k$ is a Hölder continuous rotationally symmetric function with support in $B_1 (0)$ such that $k(0) = 1$ and $(x, y) \mapsto k(x-y)$ is a covariance on $\R^d \times \R^d$. We also assume $k$ is taken to be non-negative, which is possible by the discussion in Section 2.1 of \cite{Vihko24}. Central to the argument of \cite{Vihko24} are the cut-off approximations to $S$, which are a family of coupled centered Gaussian fields with covariances
\[
	\E [S_\eps (x), S_\delta (y)] = K_{\delta, \eps} (x, y) \coloneq \int_1^{\eps^{-1} \wedge \delta^{-1}} k(s(x-y)) \frac{ds}{s} \;.
\]
As discussed in Section 2.1.1 of \cite{Vihko24}, the process $(\eps, x) \mapsto S_\eps (x)$ can be realized as a Gaussian process such that $S_\eps (x)$ is a martingale in $\eps^{-1}$ for fixed $x$, the field $S_\eps$ is almost surely Hölder continuous for fixed $\eps$, and $\langle S_\eps, f \rangle$ converges almost surely to a random variable with the same distribution as $\langle S, f \rangle$ for a fixed test function $f$. With these approximations, one can define the following auxiliary fields for $0 < \delta < \eps < 1$: \begin{align*}
	Z_{\eps, \delta, x} (u) &= S_\delta (x+\eps u) - S_\eps (x+\eps u) \;, \\
Y_{\eps, x} (u) &= S_\eps (x+\eps u) - S_\eps (x)\;.
\end{align*}
One readily checks (Proposition 2.3 of \cite{Vihko24}) that for fixed $\eps$ and $x$, $Z$ (viewed as a process of $\delta$ and $u$) is independent of both $S_\eps$ and $Y_{\eps, x}$ (viewed as processes of $u$). Moreover, we have the scale invariance property \[
	(Z_{\eps, \delta, x} (u))_{\{0 < \delta < \eps, u \in \R^d \}} \stackrel{d}{=} (S_{\delta / \eps} (u))_{\{0 < \delta < \eps, u \in \R^d \}} \;.
\]
To construct the GMC inversion map for $G$, \cite{Vihko24} first constructs it for $S$. Standard results from the theory of GMC measures (\cite{BP24}) imply that for $\gamma \in (0, \sqrt{2d})$, the measures \[
	\nu_{\gamma, \eps, S} (dx) \coloneq e^{\gamma S_\eps (x) - \frac{\gamma^2}{2} \E [S_\eps (x)^2]} \, dx
\]
converge weakly in probability to a limiting measure $\nu_{\gamma,S}$ such that $\nu_{\gamma,S} (D)$ has moments of orders $q \in (-\infty, 2d/\gamma^2)$. 

Let $\eta$ be a smooth non-negative test function on $\R^d$ with support in $B_1 (0)$ and $\int_{\R^d} \eta = 1$, and let $\eta_\eps (x) = \eps^{-d} \eta (x / \eps)$. From the definitions of the fields $Z$ and $Y$ and scale invariance, one can show the following representation (\cite{Vihko24} Section 2.3.1): \[
	\int_{\R^d} \eta_\eps (y-x) \, \nu_{\gamma, S} (dy) = e^{\gamma S_\eps (x) - \frac{\gamma^2}{2} \log (\eps^{-1})} \int_{\R^d} \eta (u) e^{\gamma Y_{\eps, x} (u)} \, \nu_{\gamma,S}^{\eps, x} (du) \;.
\]
Here $\nu_{\gamma,S}^{\eps, x}$ is the weak limit of the approximations \[
	e^{\gamma Z_{\eps, \delta, x} (u) - \frac{\gamma^2}{2} \E [Z_{\eps, \delta, x} (u)^2]} \, du
\]
on $B_1 (0)$ as $\delta \to 0$. $\nu_{\gamma,S}^{\eps, x}$ is independent of $S_\eps$ and $Y_{\eps, x}$, and is distributed like $\nu_{\gamma,S}$ on $B_1 (0)$ by scale invariance.

The GMC inversion map for $S$ can now be defined using the following deterministic function: \[
	F_{\gamma, \eps, \eta} (x) \coloneq \frac{1}{\gamma} \E \left[ \log \left( \int_{\R^d} \eta (u) e^{\gamma Y_{\eps, x} (u)} \nu_{\gamma,S}^{\eps, x} (du) \right)\right] - \frac{\gamma^2}{2} \log (\eps^{-1}) \;.
\]
The claim is that for any test function $\psi$ supported on $D$, the limit in probability as $\eps \to 0$ of \begin{align}\label{inversion-integral}
	\int_D \psi (x) \left( \frac{1}{\gamma} \log \left[ \int_{\R^d} \eta_\eps (y-x) \, \nu_{\gamma, S} (dy) \right] - F_{\gamma, \eps, \eta} 
(x) \right) \, dx
\end{align}
is $\langle S, \psi \rangle$. This would show that $S$ can be recovered from $\nu_\gamma$ on $D$ in a measurable way, as desired.

To prove the claim, \cite{Vihko24} shows that the $L^2$ distance between the integrals in \eqref{inversion-integral} and $\langle S_\eps, \psi \rangle$ decreases to zero as $\eps \to 0$.  Since $\langle S_\eps, \psi \rangle$ converges almost surely to $\langle S, \psi \rangle$, it follows that the integrals in \eqref{inversion-integral} converge in probability to $\langle S, \psi \rangle$. The difficult part is to establish that the $L^2$ distance converges to zero. For this, \cite{Vihko24} shows that $F_{\gamma, \eps, \eta} (x)$ is bounded uniformly for $x$ in the support of $\psi$, then argues carefully using dominated convergence. We defer further details of these arguments to more general case that we will consider shortly. 

To extend the inversion map for $S$ to one for $G$, \cite{Vihko24} observes (Lemma 3.2) that the GMC measure $\nu_{\gamma, G}$ satisfies \[
	\nu_{\gamma, G} (dx) = e^{\gamma H(x) - \frac{\gamma^2}{2} [g_G (x, x) - g_S (x, x)]} \nu_{\gamma, S} (dx) \;.
\]
Using this and the representation for $\nu_{\gamma, S}$, one can find an analogous representation for $\nu_{\gamma, G}$. The associated deterministic counter term turns out to be \[
	F_{\gamma, \eps, \eta} (x) + \frac{\gamma}{2} (g_S (x, x) - g_G (x, x)) \;.
\]
The only difference in the proof of convergence to $\langle G, \psi \rangle$ is that there is an additional remainder term which must be shown converges to zero. 

Now we consider what aspects of this argument must change for the slightly more general setting of Section \ref{conf-GMC}. Let $\omega$ be a measure on $D_0$ which is equivalent to Lebesgue measure with a smooth density $\lambda$ which lies in $[1-\rho, 1+\rho]$ for some $\rho < 1 \wedge (\sqrt{2d}/\gamma - 1)$. Let $G$ be a centered Gaussian field with the same covariance kernel as before, but this time with respect to the ground measure $\omega$. In other words, \[
	\E [\langle G, f_1 \rangle \langle G, f_2 \rangle] = \int_{D_0 \times D_0} f_1 (x) f_2 (y) C_G (x, y) \, \omega (dx) \, \omega (dy)
\]
where $C_G$ is as before. Theorem A of \cite{JSW19} gives a decomposition $G = S+H$ on $D$, where $S$ and $H$ are once again centered Gaussian fields, $H$ is Hölder continuous, and $S$ has covariance $C_S (x, y)$ with respect to the ground measure $\omega$. The cut-off approximations now have covariances \[
	\E [S_\eps (x) S_\delta (y)] = \lambda (x) \lambda (y) \int_1^{\eps^{-1} \wedge \delta^{-1}} k(s(x-y)) \frac{ds}{s}
\]
and the auxiliary fields $Z$ and $Y$ are defined in the same way. 

For a fixed $\eps > 0$ and $x \in D$, $Z_{\eps, \delta, x}$ is still independent from $S_\eps$ and $Y_{\eps, x}$ as before. Indeed, the covariance between $Z_{\eps, \delta, x}$ and $S_{\eps}$ can be calculated as \[
	\E [Z_{\eps, \delta, x} (u) S_\eps (v))] = \lambda (x+\eps u) \lambda (v) (K_{\delta, \eps} (x+\eps u, v) - K_{\eps, \eps} (x+\eps u, v)) = 0
\]
where the last equality follows because $\delta < \eps$, so $\delta^{-1} \wedge \eps^{-1} = \eps^{-1} \wedge \eps^{-1}$. Similarly, the covariance between $Z_{\eps, \delta, x}$ and $Y_{\eps, x}$ is
\begin{align*}
	\E[ Z_{\eps, \delta, x} (u) Y_{\eps, x} (v)] =&\lambda(x+\eps u) \lambda(x+\eps v) (K_{\delta, \eps} (x+\eps u, x+\eps v) - K_{\eps, \eps} (x+\eps u, x+\eps v)) \\
	&- \lambda (x + \eps u) \lambda (x) (K_{\delta, \eps} (x+\eps u, x) - K_{\eps, \eps} (x+\eps u, x)) = 0 \;.
\end{align*}
Moreover, $S$ is nearly scale-invariant because \begin{align*}
	\E [Z_{\eps, \delta_1, x} (u) Z_{\eps, \delta_2, x} (v)] &= \lambda (x+\eps u) \lambda (x+\eps v) (K_{\delta_1, \delta_2} (\eps u, \eps v) - K_{\eps, \eps} (\eps u, \eps v)) \\
	&= \lambda (x+\eps u) \lambda (x+\eps v) \int_{\eps^{-1}}^{\delta_1^{-1} \wedge \delta_2^{-1}} k(\eps s(u-v)) \frac{ds}{s}  \\
	&= \lambda (x+\eps u) \lambda (x+\eps v) \int_{1}^{\eps \delta_1^{-1} \wedge \eps \delta_2^{-1}} k(s(u-v)) \frac{ds}{s} \\
	&= \frac{\lambda(x+\eps u) \lambda (x+\eps v)}{\lambda(u) \lambda (v)}\E [S_{\delta_1 / \eps} (u), S_{\delta_2 / \eps} (v)] \;.
\end{align*}

Changing Lebesgue measure to $\omega$ affects the log-correlated field, but it also affects the ground measure for the GMC measure. To construct a GMC measure with a non-Lebesgue ground measure we use Shamov's theory of subcritical GMC. By Theorem 25 of \cite{Shamov16}, to construct the GMC measure associated to $S$, it suffices to show that the random variables \[
	\int_{D_0} e^{\gamma S_\eps (x) - \frac{\gamma^2}{2} \E [S_\eps (x)^2]} \, \omega(dx)
\]
are uniformly integrable. In the Lebesgue setting, uniform integrability holds when $\gamma < \sqrt{2d}$. Let $S^l_\eps$ be a family of approximations to a scale-invariant field with Lebesgue ground measure, and choose $\gamma^l \in (\gamma (1+\rho), \sqrt{2d})$, which is possible by our assumption on $\rho$. Then the covariances of our modified field $\gamma S_\eps$ are bounded above by the covariances of $\gamma^l S^l_\eps$ (note that since we took $k$ to be non-negative, all covariances are non-negative). Since uniform integrability holds for $\gamma^l S^l_\eps$ (see Section 3.3 of \cite{BP24}), it follows from Theorem 4.5.9 of \cite{Bogachev07} that there is a function $G: [0, \infty) \to [0, \infty)$ which is superlinear, convex, and increasing, such that \[
	\sup_{\eps > 0} \E \left[ G \left(\int_{D_0} e^{\gamma^l S^l_\eps (x) - \frac{(\gamma^l)^2}{2} \E [S^l_\eps (x)^2]} \, dx \right) \right] < \infty \;.
\]
By Kahane's convexity inequality (\cite{Kahane85}), we then have
\[
	\sup_{\eps > 0} \E \left[ G_0 \left(\int_{D_0} e^{\gamma S_\eps (x) - \frac{\gamma^2}{2} \E [S_\eps (x)^2]} \, \omega(dx) \right) \right] < \infty \;.
\]
where $G_0 (x) = G(x / (1 + \rho))$ to account for the change in measure. Finally, we use the other direction of Theorem 4.5.9 of \cite{Bogachev07} to deduce uniform integrability for $\gamma S_\eps$. Thus, we conclude that the measures \[
	\nu_{\gamma, \eps, S} (dx) \coloneq e^{\gamma S_\eps (x) - \frac{\gamma^2}{2} \E [S_\eps (x)^2]} \, \omega(dx)
\] 
converge weakly in probability to a limiting GMC measure $\nu_{\gamma, S}$. 

We will also need the existence of some positive moments of $\nu_{\gamma, S} (D)$. In the Lebesgue case, standard results on GMC measures imply $\E [\nu_{\gamma, S} (D)^q]$ is finite for all $q \in (-\infty, 2d/\gamma^2)$. Since we have multiplied $S_\eps$ by $\lambda$ in this new case, we can use Kahane's inequality as before to deduce that the moments still exist for $q \in (-\infty, 2d/(\gamma^2 (1+\rho)^2)$ (note that the convexity of $x^q$ fails for $q \in (0, 1)$, but these moments are guaranteed by the existence of a first moment). Since we assumed $\rho  < \sqrt{2d} / \gamma - 1$, this still guarantees moments up to some $q$ greater than $1$. In particular, $\nu_{\gamma, S} (D)$ has logarithmic moments of all orders; we will make use of this fact soon.

This measure has a similar representation to the Lebesgue case. To see this, we compute for $x \in D$ and $\eps < d(D, \partial D_0)$ that \begin{align*}
	&\int_{\R^d} \eta_\eps (y-x) \, \nu_{\gamma, S} (dy) \\
	&= \lim_{\delta \to 0} \int_{\R^d} \eta_\eps (y-x) e^{\gamma S_\delta (y) - \frac{\gamma^2}{2} \lambda(y)^2 \log (\delta^{-1})} \, \omega (dy) \\
	&=  \lim_{\delta \to 0} \int_{\R^d} \frac{\lambda (x+\eps u)}{\lambda (u)} \eta (u) e^{\gamma S_\delta (x+ \eps u) - \frac{\gamma^2}{2} \lambda (x+\eps u)^2 \log (\delta^{-1})} \, \omega (du) \\
	&= e^{\gamma S_\eps (x)} \lim_{\delta \to 0} \int_{\R^d} \frac{\lambda (x+\eps u)}{\lambda (u)} \eta (u) e^{-\frac{\gamma^2}{2} \lambda (x+\eps u)^2 \log (\eps^{-1})} e^{\gamma Y_{\eps, x} (u)} e^{\gamma Z_{\eps, \delta, x} (u) - \frac{\gamma^2}{2} \E [Z_{\eps, \delta, x} (u)^2]} \, \omega (du) \\
	&\coloneq e^{\gamma S_\eps (x)} \int_{\R^d} \frac{\lambda (x+\eps u)}{\lambda (u)} \eta (u) e^{-\frac{\gamma^2}{2} \lambda (x+\eps u)^2 \log (\eps^{-1})} e^{\gamma Y_{\eps, x} (u)} \, \nu_{\gamma, S}^{\eps, x} (du) \;.
\end{align*}
To see that the limit in the last line exists, note that by near-scale invariance it is the limit of measures with the same distributions as approximations to the GMC measure with fields $\tfrac{\lambda(x + \eps \cdot)}{\lambda (\cdot)} S_{\delta / \eps}(\cdot)$. This limit exists and possesses the same moments as $\nu_{\gamma, S}$ on $B_1 (0)$, uniformly in $x \in D$ and $\eps > 0$, by exactly the same uniform integrability argument as before. Moreover, $\nu^{\eps, x}_{\gamma, S}$ is independent of $S_\eps$ and $Y_{\eps, x}$ because the analogous independence holds for $Z_{\eps, \delta, x}$.

Before we proceed with the construction of the inverse map, we must first check a few properties of the fields $Y_{\eps, x}$ that are essential to the arguments in \cite{Vihko24}. First, we compute the covariance
\begin{align*}
	\E[Y_{\eps, x} (u) Y_{\eps, x} (v)] &=\int_1^{\eps^{-1}} \frac{\lambda(x+\eps u) \lambda(x + \eps v) k(s \eps ( u - v )) - \lambda(x+\eps u) \lambda(x) k(s \eps u) }{s} \\
	&\qquad \qquad +\frac{ -\lambda(x) \lambda(x + \eps v) k(s \eps v) + \lambda(x)^2}{s} \, ds\\
	&=\int_\eps^{1} \frac{\lambda(x+\eps u) \lambda(x + \eps v) k(s ( u - v )) - \lambda(x+\eps u) \lambda(x) k(s  u ) }{s} \\
	&\qquad \qquad +\frac{- \lambda(x) \lambda(x + \eps v) k(s v ) + \lambda(x)^2}{s} \, ds \;.
\end{align*}
We can use this to compute
\begin{align*}
	\E [\lvert Y_{\eps, x} (u) - Y_{\eps, x} (v) \rvert^2] &= \E [\lvert Y_{\eps, x} (u) \rvert^2] + \E [\lvert Y_{\eps, x} (v) \rvert^2] - 2 \E [Y_{\eps, x} (u) Y_{\eps, x} (v)] \\
	&= \int_{\eps}^1 \frac{\lambda(x+\eps u)^2 + \lambda (x+\eps v)^2 - 2\lambda(x + \eps u) \lambda (x + \eps v) k (s ( u - v ))}{s} \, ds  \\
	&= 2\lambda(x+\eps u) \lambda(x + \eps v) \int_\eps^1 \frac{1 - k (s ( u - v ))}{s} \, ds \\
	&\qquad + \int_\eps^1 \frac{\lambda (x + \eps u)^2 + \lambda (x + \eps v)^2 - 2 \lambda (x + \eps u) \lambda (x + \eps v)}{s} \, ds \\
	&\leq C \lvert u - v \rvert^{\alpha} + C' (\eps^2 \log \eps^{-1}) \lvert u - v \rvert^2 \\
	&\lesssim \lvert u - v \rvert^\alpha
\end{align*}
where $\alpha$ is the Hölder exponent of $k$ and the constants on each line are independent of $\eps > 0$.

On the other hand, for $x \neq x'$ we have
\begin{align*}
	&\E [Y_{\eps, x} (u) Y_{\eps, x'} (v)]= \\ &\int_\eps^{1} \frac{\lambda(x+\eps u) \lambda(x' + \eps v) k(s ( \eps^{-1} (x - x') + u - v )) - \lambda(x+\eps u) \lambda(x') k(s ( \eps^{-1} (x - x') + u )) }{s} \\
	&\qquad \qquad +\frac{- \lambda(x) \lambda(x' + \eps v) k(s ( \eps^{-1} (x - x') + v )) + \lambda(x) \lambda (x') k(s \eps^{-1} ( x - x' ))}{s} \, ds \\
	&= \lambda (x + \eps u) \int_{\eps}^1 \frac{(\lambda (x' + \eps v) - \lambda (x')) k(s ( \eps^{-1} (x - x') + u - v ))}{s} \, ds \\
	&\qquad + \lambda(x+\eps u) \lambda(x') \int_\eps^1 \frac{k(s ( \eps^{-1} (x - x') + u - v )) - k(s ( \eps^{-1} (x - x') + u ))}{s} \, ds \\
	&\qquad - \lambda (x) \int_\eps^1 \frac{(\lambda (x'+\eps v) - \lambda(x')) k(s ( \eps^{-1} (x - x') + v ))}{s} \, ds\\
	&\qquad - \lambda(x) \lambda (x') \int_\eps^1 \frac{k(s ( \eps^{-1} (x - x') + v )) - k(s \eps^{-1} ( x - x' ))}{s} \, ds \;.
\end{align*}
The first and third integrals are bounded by $C (\eps \log \eps^{-1}) \lvert v \rvert$ for a constant $C$ independent of $\eps$, by the smoothness of $\lambda$. The integrands in the second and fourth integrals are bounded by $C' \lvert v \rvert^\alpha s^{\alpha - 1}$ for a constant $C'$ independent of $\eps$, and converge to zero pointwise as $\eps \to 0$. In particular, one can apply dominated convergence to those two integrals, and conclude that all four integrals converge to zero as $\eps \to 0$.

There are several conclusions we can draw from these calculations. First of all, we see that the restrictions of $Y_{\eps, x}$ and $Y_{\eps, x'}$ to the unit ball converge jointly in distribution to independent centered Gaussian processes, as in Proposition 2.11 of \cite{Vihko24}. Second, we can combine this with Dudley's theorem as in Lemma 2.16 of \cite{Vihko24} to show that the supremum of $\lvert Y_{\eps, x} \rvert$ over the unit ball has all positive and exponential moments, uniformly in $x$ and $\eps$.

We can now define the counter term we will use for the GMC inversion map for $S$: \[
	F_{\gamma, \eps, \eta} (x) \coloneq \frac{1}{\gamma} \E \left[ \log \left( \int_{\R^d} \frac{\lambda (x+\eps u)}{\lambda (u)} \eta (u) \eps^{\frac{\gamma^2}{2} \lambda (x+\eps u)^2} e^{\gamma Y_{\eps, x} (u)} \, \nu_{\gamma, S}^{\eps, x} (du) \right) \right] \;.
\]

As one would expect, if $\lambda = 1$ this is exactly the $F$ that appeared in the Lebesgue case. Since $\lambda$ lies in $[1-\rho, 1+\rho]$, one can show that $F_{\gamma, \eps, \eta}$ is bounded over $D$ using exactly the same argument used to prove boundedness in Proposition 3.5 of \cite{Vihko24}, where the key ingredients are the logarithmic moments of $\nu^{\eps, x}_{\gamma, S}$ and the bounded moments of the supremum of $Y_{\eps, x}$.

As in the Lebesgue case, the remaining difficulty is now to show that integrals of the form of \eqref{inversion-integral} will be close to $\langle S_\eps, \psi \rangle$ in $L^2$ for small $\eps$. To see this, we would like to write \begin{align*}
	&\E \left[ \left( \int_D \psi(x) \left( \frac{1}{\gamma} \log \left( \int_{\R^d} \eta_\eps (y-x) \, \nu_{\gamma, S} (dy) \right) - F_{\gamma, \eps, \eta} (x)\right) \, dx - \langle S_\eps, \psi \rangle \right)^2 \right] \\
	&= \E \left[ \left( \int_D \psi(x) (A_\eps (x) - \E [A_\eps (x)]) \, dx \right)^2 \right] \\
	&= \int_{D \times D} \psi(x) \psi(x') \Cov (A_\eps (x), A_\eps (x')) \, dx  \, dx'
\end{align*}
where \[
	A_\eps (x) \coloneq \frac{1}{\gamma} \log \left( \int_{\R^d} \frac{\lambda (x+\eps u)}{\lambda (u)} \eta (u) e^{-\frac{\gamma^2}{2} (\lambda (x+\eps u)^2 - \lambda(x)^2) \log (\eps^{-1})} e^{\gamma Y_{\eps, x} (u)} \, \nu_{\gamma, S}^{\eps, x} (du) \right) \;.
\]
Here the first equality follows from the representation of $\nu_{\gamma, S}$ and the second will follow by Fubini. We have not yet justified the usage of Fubini here, but we will soon see that second moments of $A_\eps (x)$ are bounded, so Fubini will be valid. Note that the $\lambda(x)^2$ term in the exponent of $A_\eps (x)$ did not originally appear in the representation of $\nu_{\gamma, S}$, but we can include it here because it only amounts to adding to $A_\eps (x)$ a deterministic function of $\eps$ and $x$, which does not affect any of the quantities above. This term ensures that the first exponential factor in the integrand of $A_\eps (x)$ converges to $1$ as $\eps \to 0$, uniformly in $x$. 

To show that the $L^2$ distance converges to zero, it suffices by dominated convergence to show that $\Cov (A_\eps (x), A_\eps (x'))$ is bounded uniformly in $\eps > 0$ and $x, x' \in \supp \psi$, and that $\Cov (A_\eps (x), A_\eps (x')) \to 0$ as $\eps \to 0$ for any $x \neq x'$. Both of these are proven in \cite{Vihko24} in the $\lambda=1$ case, so we only have to see what changes when we allow $\lambda$ to vary from $1$.

Observe that $\lambda(x+\eps u) / \lambda(u) \in [\tfrac{1-\rho}{1+\rho}, \tfrac{1+\rho}{1-\rho}]$. Also, we saw that  $\nu^{\eps, x}_{\gamma, S}$ possesses all logarithmic moments on $B_1 (0)$, uniformly in $x$ and $\eps$. The same argument as in Lemma 3.7 of \cite{Vihko24} then shows that \[
	\E [A_\eps (x)^2] \lesssim \left( 1 + \E \left[\sup_{\lvert u \rvert \leq 1} Y_{\eps, x} (u)^2 \right] + \frac{1}{\gamma^2} \E [(\log (\nu_{\gamma, S}^{\eps, x} (\eta)))^2]\right) \;.
\]
We know that the right-hand side is bounded because of our moment bounds on the supremum of $Y_{\eps, x}$ and our logarithmic moment bounds on $\nu^{\eps, x}_{\gamma, S}$. By Cauchy-Schwarz, the covariances of the $A_\eps (x)$ are uniformly bounded above by \[
	2 \left( \sup_{x \in \supp \psi} \sup_{\eps > 0} \sqrt{\E [A_\eps (x)^2]}\right)^2 \;.
\]
We now know that this is finite, so we can conclude uniform boundedness of the covariances. In particular, this justifies our earlier use of Fubini's theorem.

Next we look at the claim about the limit of the covariances away from the diagonal as $\eps \to 0$. Recall we showed that as $\eps \to 0$, the restrictions of $Y_{\eps, x} (u)$ and $Y_{\eps, x'} (u)$ to $B_1 (0)$ converge jointly in distribution to independent centered Gaussian fields $Y_x (u)$ and $Y_{x'} (u)$. By the argument in Proposition 2.22 of \cite{Vihko24}, the measures $\nu^{\eps, x}_{\gamma, S}$ are independent on $B_1 (0)$ for different choices of $x$ and sufficiently small $\eps$. We denote by $\nu^x_{\gamma, S}$ measures which are equal in distribution to GMC measures associated to fields $\tfrac{\lambda(x)}{\lambda(\cdot)} S(\cdot)$ and independent for each $x$. Then the $\nu^{\eps, x}_{\gamma, S}$ converge weakly in distribution to $\nu^x_{\gamma, S}$ on $B_1 (0)$ by Theorem 25 of \cite{Shamov16} together with our uniform moment bounds.

We claim that the limit of the covariances of the $A_\eps (x)$ is given by the covariances of \[
	A (x) \coloneq \frac{1}{\gamma} \log \left( \int_{\R^d} \frac{\lambda (x)}{\lambda (u)} \eta(u) e^{\gamma Y_{x} (u)} \nu^x_{\gamma, S} (du) \right) \;.
\]
To see this, we first need that the pairs $(A_\eps (x), A_\eps (x'))$ converge jointly in distribution to $(A(x), A(x'))$ as $\eps \to 0$. Therefore are two key differences between our case and that of Lemma 3.7 of \cite{Vihko24} in establishing this. The first is the presence of the additional terms in the integrand that involve $\lambda$. These simply contribute additive factors to $A_\eps (x)$ and can be handled in a standard way by Slutsky's theorem, so we can ignore them and show convergence of
\[
	\frac{1}{\gamma} \log \left( \int_{\R^d} \frac{\lambda (x)}{\lambda (u)}\eta (u)e^{\gamma Y_{\eps, x} (u)} \, \nu_{\gamma, S}^{\eps, x} (du) \right)
\]
to $A(x)$, jointly in distribution. The second, more pressing issue is that the measures $\nu^{\eps, x}_{\gamma, S}$ are no longer equal in distribution for all $\eps > 0$. To see this, we proceed as in Lemma 3.7 of \cite{Vihko24} but we must work with functions that depend jointly on $Y$ and $\nu$. More precisely, \cite{Vihko24} uses that the function
\[
	f \mapsto \frac{1}{\gamma} \int_{\R^d} \eta(u) e^{\gamma f(u)} \, \mu (du)
\]
on $C(\overline{B_1 (0)})$ is almost surely continuous at $Y_x$, where $\mu$ is a fixed instantiation of $\nu^x_{\gamma, S}$. We must instead treat the above as a function of both $f$ and $\mu$, where the space of measures is equipped with the weak topology. Joint continuity is clear from the proof of continuity given in the one variable case in \cite{Vihko24} and the definition of weak convergence of measures, and the rest of the proof of joint convergence in distribution then follows in the same way as \cite{Vihko24} using the continuous mapping theorem.

To turn this into convergence of the covariances we need to swap the limit and the covariance, but this only requires a uniform bound on the fourth moment $\E [(A_\eps (x)^4]$ over $\eps$ and $x$, as explained in the proof of Lemma 3.7 of \cite{Vihko24}. This follows by a similar argument to the one for the second moment, because \[
	\E [A_\eps (x)^4] \lesssim \left( 1 + \E \left[\sup_{\lvert u \rvert \leq 1} Y_{\eps, x} (u)^4 \right] + \frac{1}{\gamma^4} \E [(\log (\nu_{\gamma, S}^{\eps, x} (\eta)))^4]\right) < \infty
\] 
uniformly in $\eps$ and $x$. The fact that $\Cov (A(x), A(x')) = 0$ for $x \neq x'$ now follows from the fact that the fields $Y_x$ and the measures $\nu^x_{\gamma, S}$ are independent for different choices of $x$. This establishes the $L^2$ convergence we wanted, and thus concludes the construction of the inverse map for $S$. 

To extend this to an inverse map for $G=S+H$, we first compute for test functions $f$ supported in $D$ that \begin{align*}
	\nu_{\gamma, G} (f) &= \lim_{\delta \to 0} \int_{\R^d} f(u) e^{\gamma G_\delta (u) - \frac{\gamma^2}{2} \E [G_\delta (u)^2]} \, \omega(du) \\
	&= \lim_{\delta \to 0} \int_{\R^d} f(u) e^{\gamma H_\delta (u)} e^{\frac{\gamma^2}{2} (\E [S_\delta (u)^2] - \E [G_\delta (u)^2])} \\
	&\qquad e^{\gamma S_\delta (u) - \frac{\gamma^2}{2} \E [S_\delta (u)^2]} \, \omega (du) \\
	&= \int_{\R^d} f(u) e^{\gamma H(u) - \frac{\gamma^2}{2} \lambda(u)^2 (g_G (u, u) - g_S (u, u))} \, \nu_{\gamma, S} (du)
\end{align*}
where $G_\delta$ is an approximation to $G$ such that such that the logarithmic divergences of the covariances of $G_\delta (u)$ and $S_\delta (u)$ cancel. This can be used to deduce the following representation for $\nu_{\gamma, G}$: \begin{align*}
	&\log \left( \int_{\R^d} \eta_\eps (y-x) \, \nu_{\gamma, G} (dy) \right) \\
	&= \gamma H(x) - \frac{\gamma^2}{2} \lambda(x)^2 g_{G, S} (x) + \log \left( \int_{\R^d} \eta_\eps (y-x) \, \nu_{\gamma, S} (dy) \right)\\
	&\qquad + \log \left( \frac{\int_{\R^d} \eta_\eps (y-x) e^{\gamma (H(y) - H(x)) - \frac{\gamma^2}{2} (\lambda(y)^2 g_{G, S} (y) - \lambda(x)^2 g_{G, S} (x))} \, \nu_{\gamma, S} (dy)}{\int_{\R^d} \eta_\eps (y-x) \, \nu_{\gamma, S} (dy)} \right)
\end{align*}
where $g_{G, S} (u) = g_G (u, u) - g_S (u, u)$. Let us denote the summand in the last line by $R_\eps (x)$. The same continuity argument as in the proof of Theorem A in \cite{Vihko24} shows that \[
	\lim_{\eps \to 0} \int_D \psi(x) R_\eps (x) \, dx = 0
\]
for any test function $\psi$. Let the counter term in this case be given by \[
	F_{\gamma, \eps, \eta, G} (x) \coloneq F_{\gamma, \eps, \eta} (x) - \frac{\gamma}{2} \lambda(x)^2 g_{G, S} (x) \;.
\]
Then we have \begin{align*}
	&\lim_{\eps \to 0} \int_D \psi(x) \left[ \frac{1}{\gamma} \log \left( \int_{\R^d} \eta_\eps (y-x) \, \nu_{\gamma, G} (dy) \right) - F_{\gamma, \eps, \eta, G} (x) \right] \, dx \\
	&= \langle S, \psi \rangle + \int_D \psi(x) H(x) \, dx + \frac{1}{\gamma} \lim_{\eps \to 0} \int_D \psi(x) R_\eps (x) \, dx = \langle G, \psi \rangle
\end{align*}
and we obtain an inverse map for $G$. More precisely, we have the following generalization of Theorem A of \cite{Vihko24}:

\begin{lemma}\label{log-GMC-generalized}
	Let $D_0 \subset \R^d$ be a bounded domain, $D$ compactly contained in $D_0$, and $\omega$ a measure on $D_0$ with smooth density $\lambda \coloneq \tfrac{d\omega}{dm}$ in $[1-\rho, 1+\rho]$ for some $\rho < \min (1, \sqrt{2d}/\gamma - 1)$, where $m$ denotes Lebesgue measure. Let $G$ be log-correlated with respect to the ground measure $\omega$, with covariance kernel $C_G (x, y) = \log (\lvert x - y \rvert)^{-1} + g_G (x, y)$ where $g_G \in H^{d+\eps}_\mathrm{loc} (D_0 \times D_0)$ is continuous on the interior of $D_0$. For $\gamma \in (0, \sqrt{2d})$, let $\nu_{\gamma, G}$ be the GMC measure associated to $G$ with respect to $\omega$. 
	
	Let $\eta$ be a test function on $\R^d$ with $\eta \geq 0$, $\int_{\R^d} \eta = 1$, and $\supp \eta \subset B_1 (0)$. Then there is a deterministic function $F_{\gamma, \eps, \eta} (x)$ such that for any test function $\psi$ supported in $D$, \[
		\int_{D_0} \psi(x) \left( \frac{1}{\gamma} \log \left[ \int_{\R^d} \eta_\eps (y-x) \, \nu_{\gamma, G} (dy) \right] - F_{\gamma, \eps, \eta} (x) \right)\, dx \to \langle G, \psi \rangle
	\]
	as $\eps \to 0$. Note that the integral above is defined only when $\eps < d(D, \partial D_0)$. In particular, this implies the existence of a measurable map $X^\gamma$ from the space of measures on $D$ to the space of distributions such that $X^\gamma (\nu_{\gamma, G}) = G$ almost surely (see Remark \ref{as-conv}).
\end{lemma}

\end{appendix}

\bibliographystyle{./Martin}
\bibliography{./refs}

\end{document}